\documentclass[12pt]{article}
\usepackage{authblk}
\usepackage[utf8]{inputenc}
\usepackage{amsmath}
\usepackage{amsthm}
\usepackage{amscd}
\usepackage{tikz}
\usetikzlibrary{decorations.markings}
\usepackage[all]{xy}
\xyoption{2cell}
\SelectTips{cm}{}
\UseAllTwocells
\usepackage{enumitem}
\usepackage{xcolor}
\usepackage{multirow}
\usepackage{pgfplots}
\pgfplotsset{compat=1.16}
\usepackage[charter]{mathdesign}
\usepackage{geometry}
\usepackage{placeins}
\usepackage{changepage}

\newtheorem{theorem}			     {Theorem}	    [section]

\newtheorem{corollary}	  [theorem]	 {Corollary}
\newtheorem{lemma}	      [theorem]  {Lemma}
\theoremstyle{definition}

\newtheorem{remark} 	  [theorem]  {Remark}
\newtheorem{example}	  [theorem]  {Example}

\newcommand{\M}{\mathsf{matr}_\ast}
\newcommand{\Mnp}{\mathsf{matr}}
\newcommand{\po}[1][dr]{\save*!/#1+1.5pc/#1:(1,-1)@^{|-}\restore}
\newdir{ >}{{}*!/-8pt/@{>}}

\title{Matrix taxonomy and Bourn localization}
\author[a,c,d]{Michael Hoefnagel}
\author[b,c]{Pierre-Alain Jacqmin}
\affil[a]{\small{\textit{Mathematics Division, Department of Mathematical Sciences, Stellenbosch University, Private Bag X1 Matieland 7602, South Africa}}}
\affil[b]{\small{\textit{Institut de Recherche en Math\'ematique et Physique, Universit\'e catholique de Louvain, Chemin du Cyclotron~2, B 1348 Louvain-la-Neuve, Belgium}}}
\affil[c]{\small{\textit{Centre for Experimental Mathematics, Department of Mathematical Sciences, Stellenbosch University, Private Bag X1 Matieland 7602, South Africa}}}
\affil[d]{\small{\textit{National Institute for Theoretical and Computational Sciences (NITheCS), South Africa}}}
\date{}

\begin{document}

\maketitle

\begin{abstract}
In a recent paper~\cite{HoefnagelJacqminJanelidze}, an algorithm has been presented for determining implications between a particular kind of category theoretic property represented by matrices --- the so called `matrix properties'. In this paper we extend this algorithm to include matrix properties involving pointedness of a category, such as the properties of a category to be unital, strongly unital or subtractive, for example. Moreover, this extended algorithm can also be used to determine whether a given matrix property is the Bourn localization of another, thus leading to new characterizations of Mal'tsev, majority and arithmetical categories. Using a computer implementation of our algorithm, we can display all such properties given by matrices of fixed dimensions, grouped according to their Bourn localizations, as well as the implications between them.
\end{abstract}



\section*{Introduction}

In this paper we are concerned with \emph{closedness properties of internal relations}~\cite{ZJanelidze2006a}, which are also simply referred to as `matrix properties'. The term `matrix property' derives itself from the presentation of these properties, which for our purposes, may always be presented simply as an extended matrix whose entries are in the set $\{\ast,x_1,x_2,\dots\}$. Here the symbols $x_1,x_2,\dots$ represent variables, and $\ast$ represents a constant symbol. For instance, the property of a finitely complete category to be a Mal'tsev category~\cite{CLP,CPP} is equivalent to the requirement that (each binary internal relation in) the category satisfies the matrix property determined by the matrix 
$$
\left[\begin{array}{ccc|c} x_1 & x_2 & x_2 & x_1 \\ 
x_2 & x_2 & x_1 & x_1 
\end{array}\right].
$$
An extended matrix, such as the one above, is said to be \emph{non-pointed} if it does not contain $\ast$ as an entry. Other examples of categorical properties which are captured by a non-pointed matrix property include the property of a category to be a majority category~\cite{Hoefnagel2019} (see also~\cite{Hoefnagel2020}) as well as the property to be a (finitely complete) arithmetical category~\cite{Pedicchio1996}. For matrices containing $\ast$'s, one also needs the category to be pointed in order to state the corresponding matrix property. Therefore, in general, each extended matrix whose entries are in the set $\{\ast,x_1,x_2,\dots\}$ gives rise to a matrix property on finitely complete pointed categories. In addition to the restriction to the pointed context of the above mentioned properties, one can cite the properties to be a unital~\cite{Bourn1996}, strongly unital~\cite{Bourn1996} or subtractive category~\cite{Janelidze2005} as examples.

In the literature, it has been shown that some matrix properties follow from a conjunction of others: a finitely complete category is arithmetical if and only if it is both Mal'tsev and a majority category (see~\cite{HoefnagelJacqminJanelidze,GRT,Hoefnagel2019}). Furthermore, a finitely complete pointed category is strongly unital if and only if it is unital and subtractive~\cite{Janelidze2005}. In the non-pointed context, an explicit algorithm has been developed in~\cite{HoefnagelJacqminJanelidze} which decides whether a conjunction of non-pointed matrix properties implies another one; however, the pointed case was not considered. The first aim of this paper is to treat this case and give an algorithm for determining whether a conjunction of matrix properties on finitely complete pointed categories implies another one. In particular, our algorithm can decide when a single matrix property implies another. We have implemented this algorithm on a computer in order to get a picture of the properties induced by matrices of relatively small dimensions and the implications between them. See for instance Figure~\ref{figure 4,3,1} representing all non-equivalent matrix properties of finitely complete pointed categories given by $4\times 3$ matrices with entries in $\{\ast,x_1\}$, as well as the relationships between them.

These algorithms allow one to compute some finite parts of the posets $\mathsf{Mclex}_\ast$ and $\mathsf{Mclex}$, i.e., the posets of collections of finitely complete pointed categories (respectively finitely complete categories) defined via matrix properties (respectively non-pointed matrix properties) and ordered by inclusion. Since each non-pointed matrix property can be restricted to the pointed context, one has an order-preserving function
$$(-)_\ast\colon\mathsf{Mclex}\to\mathsf{Mclex}_\ast.$$
Using the process of `Bourn localization' (introduced in~\cite{ZJanelidze2006b} from the ideas of~\cite{Bourn1996}), we show that this function $(-)_\ast$, regarded as a functor between preorder categories, has a right adjoint 
$$\mathsf{Loc}\colon\mathsf{Mclex}_\ast\to\mathsf{Mclex}.$$
This adjunction enables us to prove that, given two non-trivial non-pointed matrices, one has the implication between the corresponding matrix properties in the finitely complete context if and only if the analogous implication holds in the finitely complete pointed context. Moreover, it is shown in~\cite{ZJanelidze2006b} how to effectively compute the Bourn localization of a (pointed) matrix property. In our displays of the finite subposets of $\mathsf{Mclex}_\ast$ computed via the computer, we can thus group together matrix properties with the same Bourn localization, see e.g. Figure~\ref{figure 3,3,2}. See also Figure~\ref{figure delocalization Maltsev 3,4,2} which represents the subposet of $\mathsf{Mclex}_\ast$ whose elements are induced by $3\times 4$ matrices with entries in $\{\ast,x_1,x_2\}$ and such that their Bourn localization is the collection of Mal'tsev categories. This figure thus gives many more characterizations of Mal'tsev categories via the fibration of points in the style of~\cite{Bourn1996}. Moreover, in this way, we also get new characterizations of finitely complete arithmetical categories as well as majority categories (see Figures~\ref{figure delocalization arithmetical 3,5,2} and~\ref{figure delocalization majority 3,5,2}).

The paper is organized as follows:
\begin{itemize}
    \item Section~\ref{section preliminaries} recalls the relevant concepts necessary for the rest of the paper.
    \item Section~\ref{section injective objects} deals with the concept of $S$-injectivity as introduced in~\cite{HoefnagelJacqminJanelidze} in the non-pointed case.
    \item Section~\ref{section trivial matrices} characterizes those matrix properties (in the pointed context) which determine the collection of categories equivalent to the terminal category~$\mathbf{1}$ and those which determine the collection of all finitely complete pointed categories.
    \item Section~\ref{section the posets and bourn localizations} studies the posets $\mathsf{Mclex}$ and $\mathsf{Mclex}_\ast$ and the adjunction between them given by restriction and Bourn localization.
    \item Section~\ref{section the algorithm} derives the main result of this paper, which is the algorithm for determining implications between matrix properties in the finitely complete pointed context.
    \item Section~\ref{section computer-assisted results} presents some of the computer-assisted results which have been obtained based on the results of the previous sections.
\end{itemize}

Lastly, we have presented the results of this paper to be mainly self-contained. However, as our results are heavily based on~\cite{HoefnagelJacqminJanelidze}, we encourage the reader to read the introduction of~\cite{HoefnagelJacqminJanelidze}, in order to see our results here in context.

\section{Preliminaries}\label{section preliminaries}

We will sometimes denote a pointed set $(S,\ast)$ simply by $S$ if there is no ambiguity on the distinguished element $\ast$ of~$S$. Given two objects $X$ and $Y$ in a category~$\mathbb{C}$, we write as usual $\mathbb{C}(X,Y)$ for the set of all morphisms $X\to Y$. When the category $\mathbb{C}$ is pointed, i.e., when it admits a zero object, we will consider this set $\mathbb{C}(X,Y)$ as a pointed set with the distinguished element being the zero morphism $X\to Y$. As it is customary, we write $\mathbb{C}^\mathsf{op}$ for the dual of a category~$\mathbb{C}$.

Let us fix throughout this paper an infinite sequence $x_1,x_2,x_3,\dots$ of pairwise distinct variables. Given integers $n > 0$ and $m,k \geqslant 0$, we write $\M(n,m,k)$ for the set of all $n\times (m+1)$ matrices whose entries are elements of the free pointed set on $k$-variables $\{\ast,x_1,\dots,x_k\}$, and we write $\M$ for the union of all such sets. We will display such an (extended) matrix $M=[x_{ij}]_{i,j}\in\M(n,m,k)$ by
$$\left[\begin{array}{ccc|c} x_{11} & \dots & x_{1m} & x_{1\,m+1} \\ \vdots & & \vdots & \vdots \\ x_{n1} & \dots & x_{nm} & x_{n\,m+1} \end{array}\right]$$
and refer to the first $m$ columns as the \emph{left columns} of~$M$. The matrix formed by the left columns of $M$ is called the \emph{left part} of~$M$, and is denoted by~$M_\mathsf{l}$. Similarly, the last column of the matrix is called the \emph{right column} of $M$ and is denoted by~$M_\mathsf{r}$. Given a sequence $((S_1,\ast_1),\dots,(S_n,\ast_n))$ of pointed sets, a \emph{row-wise interpretation} of $M$ of \emph{type} $((S_1,\ast_1),\dots,(S_n,\ast_n))$, is an $n\times (m+1)$ matrix
$$\left[\begin{array}{ccc|c} f_1(x_{11}) & \dots & f_1(x_{1m}) & f_1(x_{1\,m+1})\\ \vdots & & \vdots & \vdots \\ f_n(x_{n1}) & \dots & f_n(x_{nm}) &  f_n(x_{n\,m+1}) \end{array}\right]$$
whose entries are obtained by applying to each row of $M$ a specified pointed function $$f_i\colon \{\ast,x_1,\dots,x_k\}\to S_i$$
where `pointed' means as usual that $f_i$ is a morphism in the category $\mathbf{Set}_\ast$ of pointed sets, i.e., that $f_i(\ast)=\ast_i$. Given a pointed set $(S,\ast)$, an (ordinary) \emph{interpretation} (i.e., a `non-row-wise' interpretation) of $M$ of \emph{type} $(S,\ast)$ is a row-wise interpretation of $M$ of type $((S,\ast),\dots,(S,\ast))$ for which $f_1=\dots=f_n$.

Given $n>0$, an \emph{internal $n$-ary relation (between objects $C_1,\dots,C_n$)} in a category $\mathbb{C}$ is given by a jointly monomorphic span $(r_i\colon R \to C_i)_{1\leqslant i \leqslant n}$, i.e., a span such that, for any two parallel morphisms $x,y\colon X\to R$ with an arbitrary domain~$X$, if $r_ix=r_iy$ for all $i\in\{1,\dots,n\}$, then $x=y$. If $C_1=\cdots =C_n=C$, we say that the $n$-ary relation $(r_i\colon R \to C)_{1\leqslant i \leqslant n}$ is a relation \emph{on} the object~$C$. When $\mathbb{C}$ has finite products, an internal $n$-ary relation in $\mathbb{C}$ can also be viewed as a monomorphism $r\colon R\rightarrowtail C_1\times\dots\times C_n$, with $r_i=\pi_ir$, where $\pi_i$ denotes $i$-th product projection $\pi_i\colon C_1\times\dots\times C_n\to C_i$. Up to identification of monomorphisms into an object with `subobjects' of that object, we can say that $n$-ary relations in $\mathbf{Set}_\ast$ between the pointed sets $(C_1,\ast_1),\dots,(C_n,\ast_n)$ are (ordinary) \emph{pointed relations}, i.e., subsets $R\subseteq C_1\times\cdots\times C_n$ containing $(\ast_1,\dots,\ast_n)$.

Given an internal $n$-ary relation $r$ of domain $R$ between $C_1,\dots,C_n$ in a pointed category~$\mathbb{C}$, a matrix $M\in\M(n,m,k)$ and an object $X$ in~$\mathbb{C}$, we say that $r$ is \emph{compatible} with a row-wise interpretation 
$$\left[\begin{array}{ccc|c} g_{11} & \dots & g_{1m} & h_1\\ \vdots & & \vdots & \vdots\\ g_{n1} & \dots & g_{nm} & h_n \end{array}\right]$$
of $M$ of type $(\mathbb{C}(X,C_1),\dots,\mathbb{C}(X,C_n))$ when, if there exist morphisms $u_1,\dots,u_m\colon X\to R$ such that 
$$\left[\begin{array}{c} g_{1j} \\ \vdots \\ g_{nj} \end{array}\right]=\left[\begin{array}{c} r_1u_j \\ \vdots \\ r_nu_j \end{array}\right]$$
for each $j\in\{1,\dots,m\}$, then there exists a morphism $v\colon X\to R$ such that
$$\left[\begin{array}{c} h_{1} \\ \vdots \\ h_{n} \end{array}\right]=\left[\begin{array}{c} r_1v \\ \vdots \\ r_nv \end{array}\right].$$
Following~\cite{ZJanelidze2006a}, we say that $r$ is \emph{strictly $M$-closed over $X$} when it is compatible with any row-wise interpretation of $M$ of type $(\mathbb{C}(X,C_1),\dots,\mathbb{C}(X,C_n))$ and we say that $r$ is \emph{strictly $M$-closed} if it is strictly $M$-closed over every object~$X$. Again according to~\cite{ZJanelidze2006a}, we say that an $n$-ary relation $(r_i\colon R \to C)_{1\leqslant i \leqslant n}$ on an object $C$ is (non-strictly) \emph{$M$-closed over} an object \emph{$X$} when it is compatible with any interpretation of $M$ of type $\mathbb{C}(X,C)$ and we say that $r$ is \emph{$M$-closed} if it is $M$-closed over every object~$X$.

\begin{theorem}\cite{ZJanelidze2006a}.
Let $n>0$ and $m,k\geqslant 0$ be integers, $M\in\M(n,m,k)$ be an extended matrix and $\mathbb{C}$ be a finitely complete pointed category. The following statement are equivalent:
\begin{itemize}
\item every $n$-ary relation in $\mathbb{C}$ is strictly $M$-closed;
\item every $n$-ary relation on an object in $\mathbb{C}$ is $M$-closed.
\end{itemize}
\end{theorem}
When the equivalent conditions in the above theorem are satisfied, we say that $\mathbb{C}$ \emph{has $M$-closed relations}.

We recall that a variety of universal algebras $\mathbb{V}$ is pointed if and only if its theory contains a unique constant term, i.e., the theory of $\mathbb{V}$ contains a nullary term $\ast$ and for any two such terms $\ast$ and $\ast'$, the theorem $\ast=\ast'$ holds in the theory. An $n$-ary internal relation in a pointed variety $\mathbb{V}$ between the algebras $A_1,\dots,A_n$ is (up to identification of `subobjects' and `monomorphisms') a pointed relation $R \subseteq A_1\times\cdots\times A_n$ compatible with the operations of~$\mathbb{V}$. Given a matrix $M\in\M(n,m,k)$, such a relation $R$ is strictly $M$-closed if and only if for any row-wise interpretation
$$\left[\begin{array}{ccc|c} a_{11} & \dots & a_{1m} & a_{1\, m+1}\\ \vdots & & \vdots & \vdots\\ a_{n1} & \dots & a_{nm} & a_{n\, m+1} \end{array}\right]$$
of $M$ of type $(A_1,\dots,A_n)$, we have:
$$\left\{\left[\begin{array}{c} a_{11} \\ \vdots \\ a_{n1} \end{array}\right],\dots, \left[\begin{array}{c} a_{1m} \\ \vdots \\ a_{nm} \end{array}\right]\right\}\subseteq  R\quad\Longrightarrow\quad \left[\begin{array}{c} a_{1\, m+1} \\ \vdots \\ a_{n\, m+1} \end{array}\right]\in R.$$
We can characterize pointed varieties of universal algebras with $M$-closed relations in the following way.

\begin{theorem}\label{theorem matrix maltsev condition}\cite{ZJanelidze2006a}.
Let $n>0$ and $m,k\geqslant 0$ be integers, $M=[x_{ij}]_{i,j}\in\M(n,m,k)$ be an extended matrix and $\mathbb{V}$ be a pointed variety of universal algebras with constant term~$\ast$. Then $\mathbb{V}$ has $M$-closed relations if and only if its theory admits an $m$-ary term $p$ satisfying the equation
$$p(x_{i1},\dots,x_{im}) = x_{i\, m+1}$$
in the variables $x_1,\dots,x_k$ for any $i\in\{1,\dots,n\}$.
\end{theorem}

For a given $M\in\M(n,m,k)$, the collection of all finitely complete pointed categories with $M$-closed relations is denoted by $\mathsf{mclex}_\ast\{M\}$ and is called a \emph{matrix class} (of finitely complete pointed categories) in this paper (this notation abbreviates the term `matrix class of left exact pointed categories' in which `left exact pointed category' is an alternative name for a `finitely complete pointed category'). Note that the word `class' here does not refer to its set-theoretic meaning, since such a matrix class could not be a class in the set-theoretic sense. By a \emph{matrix property} (of a finitely complete pointed category) we mean here the property to have $M$-closed relations for a given matrix~$M$ (as in~\cite{JanelidzePhD}).

\begin{remark}
Let us warn the reader that the way we display matrices in $\M(n,m,k)$ here is slightly different from how they are displayed in~\cite{ZJanelidze2006a,ZJanelidze2006b} where $0$ is used instead of~$\ast$, or from \cite{HoefnagelJacqminJanelidze} where integers are used instead of variables and where the right columns have been omitted. Indeed, in~\cite{HoefnagelJacqminJanelidze}, the right column could be made the same for all matrices without changing the corresponding properties on finitely complete categories. They could therefore be omitted. In the present paper, the right column can contain variables and $\ast$'s and thus they cannot, a priori, be made all the same. This explains why we have to write them explicitly here. Also, we have preferred $\ast$'s instead of $0$'s not to confuse them with the $0$'s from~\cite{HoefnagelJacqminJanelidze}. Finally, the choice of variables instead of integers has been made to stay closer to the original notation as can be found in~\cite{ZJanelidze2006a}.
\end{remark}

We say that a matrix $M\in\M(n,m,k)$ is \emph{non-pointed} if it does not have $\ast$ as one of its entries, and denote the subset of $\M$ consisting of all non-pointed matrices by~$\Mnp$. Borrowing notation from~\cite{HoefnagelJacqminJanelidze}, we write $\Mnp(n,m,k)$ for the set $\M(n,m,k) \cap \Mnp$. For a non-pointed matrix~$M$, the notion of a category with $M$-closed relations extends to the (not necessarily pointed) finitely complete context in the obvious way. As in~\cite{HoefnagelJacqminJanelidze}, we denote by $\mathsf{mclex}\{M\}$ the collection of finitely complete categories with $M$-closed relations. Therefore, for a non-pointed matrix~$M$, $\mathsf{mclex}_\ast\{M\}$  is the sub-collection of $\mathsf{mclex}\{M\}$ consisting of pointed categories in $\mathsf{mclex}\{M\}$.

By a \emph{matrix set} we simply mean a subset $S$ of~$\M$. We denote by $\mathsf{mclex}_\ast S$ the collection of all finitely complete pointed categories which have $M$-closed relations for each matrix $M$ in~$S$, i.e., the intersection $\bigcap_{M\in S}\mathsf{mclex}_\ast\{M\}$. This extends the notion of a matrix set from~\cite{HoefnagelJacqminJanelidze} where only matrix sets containing non-pointed matrices were considered (and up to the change of notation as explained above). We thus say that a matrix set is \emph{non-pointed} if it is a subset of~$\Mnp$. For a non-pointed matrix set~$S$, we denote by $\mathsf{mclex} S$ the collection of all finitely complete categories which have $M$-closed relations for each matrix $M$ in~$S$, i.e., the intersection $\bigcap_{M\in S}\mathsf{mclex}\{M\}$.

\begin{example}\label{example matrices}
The following table shows on the left some examples of matrices and on the right the corresponding matrix classes. We denote by $\mathbf{1}$ the category with a single object and a single morphism. The references indicates where the corresponding exactness properties have been introduced, possibly in a different context than that of finitely complete pointed categories.

\begin{center}
\begin{tabular}{|c|c|}
\hline \rule{0pt}{3ex}\rule[-1.5ex]{0pt}{0pt} Extended matrix & Matrix class \\ \hline
\rule{0pt}{4ex}\rule[-2.7ex]{0pt}{0pt} $M=\left[\begin{array}{ccc|c} x_1 & x_2 & x_2 & x_1\\ x_2 & x_2 & x_1 & x_1 \end{array}\right]$ & (pointed) Mal'tsev categories~\cite{CLP,CPP} \\ \hline
\rule{0pt}{6ex}\rule[-4.5ex]{0pt}{0pt} $M=\left[\begin{array}{ccc|c} x_2 & x_1 & x_1 & x_1\\ x_1 & x_2 & x_1 & x_1 \\ x_1 & x_1 & x_2 & x_1 \end{array}\right]$ & (pointed) majority categories~\cite{Hoefnagel2019} \\ \hline
\rule{0pt}{6ex}\rule[-4.5ex]{0pt}{0pt} $M=\left[\begin{array}{ccc|c} x_1 & x_2 & x_2 & x_1\\ x_2 & x_2 & x_1 & x_1 \\ x_1 & x_2 & x_1 & x_1 \end{array}\right]$ & (finitely complete pointed) arithmetical categories~\cite{Pedicchio1996} \\ \hline
\rule{0pt}{4ex}\rule[-2.7ex]{0pt}{0pt} $M=\left[\begin{array}{cc|c} x_1 & \ast & x_1\\ \ast & x_1 & x_1 \end{array}\right]$ & unital categories~\cite{Bourn1996} \\ \hline
\rule{0pt}{4ex}\rule[-2.7ex]{0pt}{0pt} $M=\left[\begin{array}{ccc|c} x_1 & \ast & \ast & x_1\\ x_2 & x_2 & x_1 & x_1 \end{array}\right]$ & strongly unital categories~\cite{Bourn1996} \\ \hline
\rule{0pt}{4ex}\rule[-2.7ex]{0pt}{0pt} $M=\left[\begin{array}{cc|c} x_1 & \ast & x_1\\ x_1 & x_1 & \ast \end{array}\right]$ & subtractive categories~\cite{Janelidze2005} \\ \hline
\rule{0pt}{3ex}\rule[-1.5ex]{0pt}{0pt} $M=\left[\begin{array}{c|c} \, & x_1\end{array}\right]$ & categories equivalent to~$\mathbf{1}$ \\ \hline
\rule{0pt}{3ex}\rule[-1.5ex]{0pt}{0pt} $M=\left[\begin{array}{c|c} \, & \ast\end{array}\right]$ & all finitely complete pointed categories \\ \hline
\end{tabular}\vspace{3pt}
\end{center}
\end{example}

\begin{example}
According to Theorem~\ref{theorem matrix maltsev condition}, the matrix properties in Example~\ref{example matrices} all determine a corresponding Mal'tsev condition for pointed varieties of universal algebras. For instance, the matrix corresponding to the collection of subtractive categories determines a binary operation $s(x,y)$ which satisfies the equations $s(x,\ast) = x$ and $s(x,x) = \ast$, i.e., a \emph{subtraction}. On the other hand, many well known Mal'tsev conditions may be strengthened to a Mal'tsev condition which is determined by a matrix property. This process has been described in~\cite{HoefnagelJacqminJanelidze} (see Remark~1.8 therein) and was termed \emph{syntactical refinement}. To illustrate this process in the pointed context, we consider the Mal'tsev condition obtained in \cite{Hoefnagel2019b} corresponding to the commutativity of binary products with coequalizers. For a pointed variety $\mathbb{V}$ (where $\ast$ is the unique constant symbol) this Mal'tsev condition asserts that there exists integers $m\geqslant 0$ and $n\geqslant 1$ such that  $\mathbb{V}$ admits binary terms $b_{i}(x,y)$ and unary terms $c_{i}(x)$ for each $1 \leqslant i \leqslant m$ and $(m+2)$-ary terms $p_1,p_2,\dots,p_n$ satisfying the equations: 
\begin{align*}
&p_1(x,y,b_{1}(x,y),\dots,b_{m}(x,y)) = x, \\
&p_i(y,x,b_{1}(x,y),\dots,b_{m}(x,y)) = p_{i+1}(x,y,b_{1}(x,y),\dots,b_{m}(x,y)) \quad\text{for each }i\in\{1,\dots,n-1\},\\
&p_n(y,x,b_{1}(x,y),\dots,b_{m}(x,y)) = y,\\
&p_i(\ast,\ast,c_1(x),\dots,c_{m}(x)) = x\quad\text{for each }i\in\{1,\dots,n\}.
\end{align*}
We may strengthen this condition by assuming that $n=1$ and letting $p=p_1$, then the resulting equations reduce to
\begin{align*}
p(x,y,b_{1}(x,y),\dots,b_{m}(x,y)) &= x, \\
p(y,x,b_{1}(x,y),\dots,b_{m}(x,y)) &= y, \\ 
p(\ast,\ast,c_1(x),\dots,c_{m}(x)) &= x.
\end{align*}
This condition may be further strengthened by assuming that the terms $b_i(x,y)$ and $c_i(x)$ are actually variables so that $m=2$, $b_1(x,y) = x$, $b_2(x,y) = y$ and $c_1(x)=c_2(x)=x$. Doing this reduces the above equations to
\begin{align*}
p(x,y,x,y) &= x, \\
p(y,x,x,y) &= y, \\ 
p(\ast,\ast,x,x) &= x.
\end{align*}
These equations, in turn, determine a matrix property which is determined by the matrix
$$M=\left[\begin{array}{cccc|c}
x_1  & x_2  & x_1 & x_2 & x_1 \\ 
x_2  & x_1  & x_1 & x_2 & x_2 \\ 
\ast & \ast & x_1 & x_1 & x_1 \\ 
\end{array}\right].$$
This matrix property is then such that any pointed variety of universal algebras that satisfies it necessarily has that binary products commute with coequalizers. This procedure of syntactical refinement can then be applied to other Mal'tsev conditions which apply to pointed varieties, such as for instance \emph{anticommutativity} in the sense of~\cite{Hoefnagel2021} or \emph{normality of projections} in the sense of~\cite{ZJanelidze2003}. In the first case, the syntactical refinement of the Mal'tsev condition appearing in Theorem~3.1 of~\cite{Hoefnagel2021} is only satisfied by varieties of universal algebras equivalent to the terminal category~$\mathbf{1}$ (i.e., the corresponding matrix is trivial in the sense of Section~\ref{section trivial matrices}). In the second case, syntactical refinement of the Mal'tsev condition appearing in Theorem~3 of~\cite{ZJanelidze2003} yields the matrix
$$\left[\begin{array}{ccc|c}
x_1 & x_1  & \ast & x_1  \\ 
x_1 & \ast & x_1  & \ast \\ 
x_1 & \ast & \ast & x_1  \\ 
\end{array}\right]$$
whose corresponding Mal'tsev condition on pointed varieties of universal algebras implies the normality of projections. Using the algorithm presented in this paper, one can show that the matrix class determined by this matrix is the collection of subtractive categories, which illustrates the fact that subtractive categories have normal projections~\cite{Janelidze2005}.
\end{example}

According to Proposition~1.7 in~\cite{ZJanelidze2006b}, we know that:
\begin{itemize}
\item Given two matrices $M\in\M(n,m,k)$ and $N\in\M(n',m,k')$ such that every row of $N$ is a row of~$M$, then any finitely complete pointed category with $M$-closed relations also has $N$-closed relations, i.e., $\mathsf{mclex}_\ast\{M\} \subseteq \mathsf{mclex}_\ast\{N\}$.
\item Given two matrices $M\in\M(n,m,k)$ and $N\in\M(n,m',k')$ such that every left column of $M$ is a left column of $N$ and the right column of $M$ is the same as the right column of~$N$, then any finitely complete pointed category with $M$-closed relations also has $N$-closed relations, i.e., $\mathsf{mclex}_\ast\{M\} \subseteq \mathsf{mclex}_\ast\{N\}$.
\end{itemize}
It follows immediately from these statements that matrix properties are invariant under duplication and permutation of left columns and of rows of the matrices. It is also clear that the matrix property arising from $M\in\M(n,m,k)$ is the same as the one arising from $M$ viewed in $\M(n,m,k')$ for any $k'\geqslant k$. It is also evident that matrix properties are invariant under addition of a left column of~$\ast$'s and permutation of the variables $x_1,\dots,x_k$ in a specified row. 

The notions of $M$-closedness and strict $M$-closedness have been studied in~\cite{ZJanelidze2006a,ZJanelidze2006b}. A third type of closedness property has been introduced in~\cite{HoefnagelJacqminJanelidze}. Given a matrix $M\in\M(n,m,k)$, we say that an internal $n'$-ary relation in a pointed category is \emph{$M$-sharp} if it is strictly $M'$-closed under any matrix $M'$ obtained from any selection of $n'$ rows from~$M$. Note that in this notion, the number $n$ of rows of $M$ is not required to be exactly~$n'$; one can have $n>n'$, $n=n'$ or $n<n'$. By the above remark, a finitely complete pointed category has $M$-closed relations if and only if every $n'$-ary internal relation in it is $M$-sharp for any positive integer~$n'$. We also have a designated concept of an interpretation for the closedness property of sharpness: given pointed sets $(S_1,\ast_1),\dots,(S_{n'},\ast_{n'})$, we define a \emph{reduction} of type $((S_1,\ast_1),\dots,(S_{n'},\ast_{n'}))$ of a matrix $M\in\M(n,m,k)$ to be a row-wise interpretation of type $((S_1,\ast_1),\dots,(S_{n'},\ast_{n'}))$ of some matrix $M'\in\M(n',m,k)$ such that every row of $M'$ is a row of~$M$, and where duplicate left columns of the row-wise interpretation have been (possibly) deleted as well as left columns (possibly) permuted or duplicated. Then an internal $n'$-ary relation between the objects $C_1,\dots,C_{n'}$ in a pointed category $\mathbb{C}$ is $M$-sharp if and only if it is compatible with every reduction of $M$ of type $(\mathbb{C}(X,C_1),\dots,\mathbb{C}(X,C_{n'}))$ (for an arbitrary object~$X$).

\begin{example}
To illustrate the concept of $M$-sharpness, let us consider the matrix
$$M=\left[\begin{array}{ccc|c} x_1 & \ast & \ast & x_1\\ x_2 & x_2 & x_1 & x_1 \end{array}\right]$$
from Example~\ref{example matrices} representing strongly unital categories. A binary relation is $M$-sharp if it is strictly $M_{ij}$-closed for all $i,j\in\{1,2\}$ where
$$\begin{array}{lr}
M_{11}=\left[\begin{array}{ccc|c} x_1 & \ast & \ast & x_1 \\ x_1 & \ast & \ast & x_1 \end{array}\right] & M_{12}=M=\left[\begin{array}{ccc|c} x_1 & \ast & \ast & x_1 \\ x_2 & x_2 & x_1 & x_1 \end{array}\right]\hspace{5pt}\\
M_{21}=\left[\begin{array}{ccc|c} x_2 & x_2 & x_1 & x_1 \\ x_1 & \ast & \ast & x_1 \end{array}\right] & M_{22}=\left[\begin{array}{ccc|c} x_2 & x_2 & x_1 & x_1 \\ x_2 & x_2 & x_1 & x_1 \end{array}\right].
\end{array}$$
Since the right column of $M_{11}$ can be found in its left columns (and analogously for~$M_{22}$), any binary relation is strictly $M_{11}$-closed and strictly $M_{22}$-closed. Hence, a binary relation is $M$-sharp if and only if it is strictly $M$-closed and strictly $M_{21}$-closed. In the category $\mathbf{Set}_\ast$, let us consider the pointed binary relation
$$R=\left\{\left[\begin{array}{c}\ast \\ \ast\end{array}\right],\left[\begin{array}{c}a \\ \ast\end{array}\right],\left[\begin{array}{c}a \\ a\end{array}\right]\right\}\subset A^2$$
on the two-element pointed set $A=\{\ast,a\}$. This relation is strictly $M$-closed since, given two pointed functions $f_1,f_2\colon\{\ast,x_1,x_2\}\to\{\ast,a\}$ such that
$$\left[\begin{array}{c}f_1(x_1) \\ f_2(x_2)\end{array}\right],\left[\begin{array}{c}f_1(\ast) \\ f_2(x_2)\end{array}\right],\left[\begin{array}{c}f_1(\ast) \\ f_2(x_1)\end{array}\right] \in R,$$
one automatically has $f_2(x_1)=f_2(x_2)=\ast$ and therefore
$$\left[\begin{array}{c}f_1(x_1) \\ f_2(x_1)\end{array}\right] \in R.$$
Moreover, since $R$ is reflexive, we can see that it is (non-strictly) $M_{21}$-closed. However, $R$ is not strictly $M_{21}$-closed since
$$\left[\begin{array}{ccc|c} a & a & \ast & \ast \\ a & \ast & \ast & a \end{array}\right]$$
is a row-wise interpretation of $M_{21}$ of type $(A,A)$ whose left columns are in $R$ but not its right column. Therefore, R is not $M$-sharp which show that, in general, when the arity $n'$ of the relation $R$ on an object is the same as the number $n$ of rows of the matrix $M\in\M$, the condition `$R$ is $M$-sharp' is strictly stronger than the condition `$R$ is strictly $M$-closed', which is itself strictly stronger than the condition `$R$ is $M$-closed'. Note that when $n\neq n'$, the comparison between $M$-sharpness and (strict) $M$-closedness does not make sense.
\end{example}

\section{$S$-injective objects}\label{section injective objects}

Given a matrix set~$S$ and extending the notions from~\cite{HoefnagelJacqminJanelidze}, we call an object $X$ in a pointed category $\mathbb{C}$ an \emph{$S$-injective object} if, for every matrix $M\in S$, every internal $n$-ary relation on any object in $\mathbb{C}^\mathsf{op}$ is $M$-closed over~$X$, where $n$ is the number of rows of~$M$. If $\mathbb{C}$ is finitely cocomplete (i.e., if $\mathbb{C}^\mathsf{op}$ is finitely complete), this is equivalent to require that, for every matrix $M\in S$, every internal $n$-ary relation in $\mathbb{C}^\mathsf{op}$ is strictly $M$-closed over~$X$, where $n$ is the number of rows of~$M$. It is not difficult to see that $S$-injective objects are closed under all limits that exist in the pointed category. We write $\mathsf{Inj}_S\mathbb{C}$ for the full subcategory of $\mathbb{C}$ consisting of all $S$-injective objects. We say that $\mathbb{C}$ \emph{has enough $S$-injective objects} if for any object $C$ in $\mathbb{C}$ there is a monomorphism $C\rightarrowtail D$ for which $D$ is $S$-injective.

Given a matrix $M=[x_{ij}]_{i,j}\in\M(n,m,k)$, a pointed category $\mathbb{C}$ with finite products and finite coproducts, and an object $X$ in~$\mathbb{C}$, we denote by $\pi^X_{M_\mathsf{l}}$ the canonical map from the coproduct $n X^k$ to the product $X^m$ given by the matrix
$$\pi^X_{M_\mathsf{l}}=\left[\begin{array}{ccc} \pi_{x_{11}} & \dots & \pi_{x_{1m}} \\ \vdots & & \vdots \\ \pi_{x_{n1}} & \dots & \pi_{x_{nm}} \end{array}\right]$$
where $\pi_\ast$ stands for the zero morphism $X^k\to X$ and $\pi_{x_1},\dots,\pi_{x_k}$ stand for the product projections $X^k\to X$. Similarly, we denote by $\pi^X_{M_\mathsf{r}}$ the canonical map from the coproduct $n X^k$ to $X$ given by the column
$$\pi^X_{M_\mathsf{r}}=\left[\begin{array}{c} \pi_{x_{1\,m+1}} \\ \vdots \\ \pi_{x_{n\,m+1}} \end{array}\right].$$

We say that the category $\mathbb{C}$ has \emph{universal epi-factorizations} if every morphism $f$ admits a factorization $f=gh$ via an epimorphism $h$ such that any similar factorization $f = g'h'$ with $h'$ an epimorphism yields $h = uh'$ for a (necessarily unique) morphism~$u$. Such a factorization is of course unique up to isomorphism and we call it the \emph{universal epi-factorization} of~$f$. The following theorem is the `pointed version' of Theorem~2.2 of~\cite{HoefnagelJacqminJanelidze}.

\begin{theorem}\label{theorem description injective objects}
Let $n>0$ and $m,k\geqslant 0$ be integers, $M\in\M(n,m,k)$ be a matrix, $\mathbb{C}$ be a pointed category with finite products, finite colimits and universal epi-factorizations, and $X$ be an object of~$\mathbb{C}$. Considering the universal epi-factorization $\pi^X_{M_\mathsf{l}}=i^X_{M_\mathsf{l}} r^X_{M_\mathsf{l}}$ of~$\pi^X_{M_\mathsf{l}}$, the following statements are equivalent:
\begin{enumerate}[label=(\arabic*)]
\item\label{X is M-injective} $X$ is $\{M\}$-injective.
\item\label{relations strictly M-closed over X} Every internal $n$-ary relation in $\mathbb{C}^\mathsf{op}$ is strictly $M$-closed over~$X$.
\item\label{description injective objects} There exists a morphism $p^X_M\colon R^X_{M_\mathsf{l}} \to X$ making the diagram
$$\xymatrix{& X^m &\\ & & n X^k\ar@{->>}[ld]_-{r^X_{M_\mathsf{l}}}\ar[rd]^-{\pi^X_{M_\mathsf{r}}}\ar[lu]_-{\pi^X_{M_\mathsf{l}}} & \\ & R^X_{M_\mathsf{l}}\ar[uu]^-{i^X_{M_\mathsf{l}}} \ar@{-->}[rr]_-{p^X_M} & & X}$$
commute.
\end{enumerate}
\end{theorem}

\begin{proof}
As mentioned above, the equivalence \ref{X is M-injective}$\Leftrightarrow$\ref{relations strictly M-closed over X} holds because $\mathbb{C}$ is finitely cocomplete. To prove the implication \ref{X is M-injective}$\Rightarrow$\ref{description injective objects}, we consider the epimorphism $r^X_{M_\mathsf{l}}\colon n X^k \twoheadrightarrow R^X_{M_\mathsf{l}}$ as an $n$-ary relation in~$\mathbb{C}^\mathsf{op}$. We also consider the interpretation of $M$ of type $\mathbb{C}^\mathsf{op}(X,X^k)=\mathbb{C}(X^k,X)$ given by the pointed function $f\colon\{\ast,x_1,\dots,x_k\}\to\mathbb{C}(X^k,X)$ defined by $f(x_i)=\pi_{x_i}$, the $i$-th projection $X^k\to X$, for each $i\in\{1,\dots,k\}$. Since $\pi^X_{M_\mathsf{l}}$ factorizes through $r^X_{M_\mathsf{l}}$ and since this relation is $M$-closed over $X$ in~$\mathbb{C}^\mathsf{op}$, we know that $\pi^X_{M_\mathsf{r}}$ must also factorize through $r^X_{M_\mathsf{l}}$.

It remains to prove the implication \ref{description injective objects}$\Rightarrow$\ref{X is M-injective}. Let us consider an $n$-ary internal relation on an object $C$ in~$\mathbb{C}^\mathsf{op}$, viewed as an epimorphism $r\colon nC\twoheadrightarrow R$ in~$\mathbb{C}$. Let us also consider an interpretation of $M$ of type $\mathbb{C}^\mathsf{op}(X,C)=\mathbb{C}(C,X)$ given by a pointed function $f\colon\{\ast,x_1,\dots,x_k\}\to\mathbb{C}(C,X)$. This function $f$ induces a morphism $g\colon C\to X^k$ such that $\pi_{x_i}g=f(x_i)$ for each $i\in\{1,\dots,k\}$. We denote by $ng\colon nC\to n X^k$ the $n$-fold coproduct of the morphism~$g$. Given a factorisation of $\pi^X_{M_\mathsf{l}}\circ ng$ through $r$ as in
$$\xymatrix{X^m & &\\ & & n X^k\ar[rd]^-{\pi^X_{M_\mathsf{r}}}\ar[llu]_-{\pi^X_{M_\mathsf{l}}} & \\ & nC\ar@{->>}[dl]_{r}\ar[ur]_-{ng} & & X \\ R\ar@{-->}[urrr]\ar[uuu] & }$$
we must show that $\pi^X_{M_\mathsf{r}}\circ ng$ also factors through~$r$. Considering the pushout of $r$ along~$ng$,
$$\xymatrix{nC \ar@{->>}[d]_-{r} \ar[r]^-{ng} & n X^k \ar@{->>}[d]^-{r'} \\ R \ar[r] & R' \po}$$
it suffices to show that $\pi^X_{M_\mathsf{r}}$ factors through $r'$ under the assumption that $\pi^X_{M_\mathsf{l}}$ does.
$$\xymatrix{& X^m &\\ & & n X^k\ar@{->>}[ld]_-{r'}\ar[rd]^-{\pi^X_{M_\mathsf{r}}}\ar[lu]_-{\pi^X_{M_\mathsf{l}}} & \\ & R'\ar[uu]\ar@{-->}[rr] & & X}$$
Since $\pi^X_{M_\mathsf{l}}=i^X_{M_\mathsf{l}} r^X_{M_\mathsf{l}}$ is a universal epi-factorization, there exists a morphism $u\colon R'\to R^X_{M_\mathsf{l}}$ such that $ur' = r^X_{M_\mathsf{l}}$. The required morphism $R'\to X$ is then given by the composite $p^X_M u$.
\end{proof}

We will later need the following result, which is the direct adaptation to the pointed context of Theorem~3.1 in~\cite{HoefnagelJacqminJanelidze}, itself coming from the ideas of~\cite{Weighill2017}. The proof being analogous to the one in~\cite{HoefnagelJacqminJanelidze}, we omit it here.

\begin{theorem}\label{ThmB}
Let $\mathbb{C}$ be a pointed category having finite limits and finite colimits and where every morphism factorizes as an epimorphism followed by an equalizer. Consider two matrix sets $S\subseteq T$. If $\mathbb{C}$ has enough $T$-injective objects, then $\mathsf{Inj}_S\mathbb{C}$ is the largest full subcategory of $\mathbb{C}$ among those that contain all $T$-injective objects, are closed under finite limits, and whose dual categories have $M$-closed relations for every $M\in S$.
\end{theorem}

\section{Trivial and anti-trivial matrices}\label{section trivial matrices}

A matrix $M\in\M$ is said to be \emph{trivial} if the matrix class $\mathsf{mclex}_\ast\{M\}$ consists exactly of those categories which are equivalent to~$\mathbf{1}$, the single morphism category. Since those categories belong to $\mathsf{mclex}_\ast\{N\}$ for all $N\in\M$, one has $\mathsf{mclex}_\ast\{M\} \subseteq \mathsf{mclex}_\ast\{N\}$ for any trivial matrix $M$ and any matrix~$N$. This section is devoted to the proof of Theorem~\ref{ThmC} which characterizes trivial matrices. In view of that theorem and of Theorem~2.3 in~\cite{HoefnagelJacqminJanelidze}, we know that a non-pointed matrix $M\in\Mnp\subset\M$ is trivial in our sense if and only if it is trivial in the sense of~\cite{HoefnagelJacqminJanelidze}, i.e., if and only if any finitely complete category with $M$-closed relations is a preorder (see Corollary~\ref{corollary trivial matrices}).

Given a matrix $M=[x_{ij}]_{i,j}\in\M(n,m,k)$ and an integer $i\in\{1,\dots,n\}$, we denote by $R_{M_i}$ the kernel relation on the set $\{1,\dots,m\}$ induced by the left part of the $i$-th row of $M$ seen as a map $\{1,\dots,m\}\to\{\ast,x_1,\dots,x_k\}$, i.e.,
$$j\, R_{M_i} \,j' \,\,\Leftrightarrow \,\, x_{ij}=x_{ij'}.$$
Given $i,i'\in\{1,\dots,n\}$, we denote by $R^\ast_{M_{i,i'}}$ the equivalence relation on the set $\{1,\dots,m\}$ defined by
$$j\, R^\ast_{M_{i,i'}} \,j' \,\,\Leftrightarrow \,\, (j=j') \text{ or } ((x_{i,j}=\ast \text{ or } x_{i',j}=\ast) \text{ and } (x_{i,j'}=\ast \text{ or } x_{i',j'}=\ast))$$
We also consider the equivalence relation $R_{M_i} \vee R_{M_{i'}} \vee R^\ast_{M_{i,i'}}$ given by the join of the three equivalence relations $R_{M_i}$, $R_{M_{i'}}$ and $R^\ast_{M_{i,i'}}$. Finally, given a pointed set~$(Y,\ast)$, we say that $M$ is \emph{functional in $(Y,\ast)$} if, given $i,i'\in\{1,\dots,n\}$ and two pointed functions
$$f_i,f_{i'}\colon\{\ast,x_1,\dots,x_k\}\to Y$$
such that $f_i(x_{ij})=f_{i'}(x_{i'j})$ for all $j\in\{1,\dots,m\}$, one has $f_i(x_{i\,m+1})= f_{i'}(x_{i'\,m+1})$. Using Theorem~\ref{theorem description injective objects} with $\mathbb{C}=\mathbf{Set}_\ast$ and $X=(Y,\ast)$, since $R^X_{M_\mathsf{l}}$ is in that case the image of $\pi^{(Y,\ast)}_{M_\mathsf{l}}$, it is not difficult to see that $M$ is functional in $(Y,\ast)$ if and only if $(Y,\ast)$ is an $\{M\}$-injective object in $\mathbf{Set}_\ast$.

\begin{theorem}\label{ThmC}
Given integers $n>0$ and $m,k\geqslant 0$ and a matrix $M=[x_{ij}]_{i,j}\in\M(n,m,k)$, the following conditions are equivalent:
\begin{enumerate}[label=(\arabic*)]
\item\label{ThmC M not trivial} $M$ is not a trivial matrix.

\item\label{ThmC Set op M-closed relations} $\mathbf{Set}_\ast^\mathsf{op}$ has $M$-closed relations.

\item\label{ThmC description} The following three conditions are all satisfied:
\begin{itemize}
\item every row of $M$ whose right entry is $x_{i\,m+1}\neq\ast$ contains $x_{i\,m+1}$ as one of its left entries,
\item given $i,i'\in\{1,\dots,n\}$ and $j,j'\in\{1,\dots,m\}$ such that $i\neq i'$ and $$x_{ij}=x_{i\,m+1}\neq\ast\neq x_{i'j'}=x_{i'\,m+1},$$
then $j\,R_{M_i} \vee R_{M_{i'}} \vee R^\ast_{M_{i,i'}} \,j'$,
\item given $i,i'\in\{1,\dots,n\}$ and $j\in\{1,\dots,m\}$ such that $x_{ij}=x_{i\,m+1}\neq\ast=x_{i'\,m+1}$, there exists $j'\in\{1,\dots,m\}$ such that $j\,R_{M_i} \vee R_{M_{i'}} \vee R^\ast_{M_{i,i'}} \,j'$ and either $x_{ij'}=\ast$ or $x_{i'j'}=\ast$.
\end{itemize}

\item\label{ThmC reductions} $M$ does not have a reduction of type $(\{\ast,x_1\})$ or $(\{\ast,x_1\},\{\ast,x_1\})$ given by any of the following four matrices:
$$\left[\begin{array}{c|c} \, & x_1\end{array}\right],\quad \left[\begin{array}{c|c} \ast & x_1\end{array}\right],\quad \left[\begin{array}{c|c} x_1 & x_1\\ x_1 & \ast\end{array}\right],\quad \left[\begin{array}{cc|c} x_1 & \ast & x_1\\ x_1 & \ast & \ast\end{array}\right].$$

\item\label{ThmC functionality all} $M$ is functional in every pointed set.

\item\label{ThmC functionality 2 elements} $M$ is functional in a two element pointed set.

\item\label{ThmC functionality at least 2 elements} $M$ is functional in a pointed set having at least two elements.

\end{enumerate} 
\end{theorem}

\begin{proof}
The implications \ref{ThmC functionality all}$\Rightarrow$\ref{ThmC functionality 2 elements}$\Rightarrow$\ref{ThmC functionality at least 2 elements} are obvious. Let us show \ref{ThmC functionality at least 2 elements}$\Rightarrow$\ref{ThmC reductions}. Suppose \ref{ThmC reductions} does not hold and let $(Y,\ast)$ be a pointed set and $y\in Y$ with $y\neq\ast$. If $$\left[\begin{array}{c|c} \, & x_1\end{array}\right]\quad \left(\text{respectively }\left[\begin{array}{c|c} \ast & x_1\end{array}\right],\quad \left[\begin{array}{c|c} x_1 & x_1\\ x_1 & \ast\end{array}\right],\text{ or } \left[\begin{array}{cc|c} x_1 & \ast & x_1\\ x_1 & \ast & \ast\end{array}\right]\right)$$
is a reduction of~$M$, so is
$$\left[\begin{array}{c|c} \, & y\\ \, & \ast\end{array}\right]\quad \left(\text{respectively }\left[\begin{array}{c|c} \ast & y \\ \ast & \ast\end{array}\right],\quad \left[\begin{array}{c|c} y & y\\ y & \ast\end{array}\right], \text{ or }\left[\begin{array}{cc|c} y & \ast & y\\ y & \ast & \ast\end{array}\right]\right)$$
showing that $M$ is not functional in~$(Y,\ast)$.

Let us now prove \ref{ThmC reductions}$\Rightarrow$\ref{ThmC description}. Suppose \ref{ThmC description} does not hold and let us show that \ref{ThmC reductions} does not hold neither. If $m=0$, this means that $M$ has a row with $x_{i\,m+1}\neq\ast$ as right entry (and empty left part) and therefore $M$ admits
$$\left[\begin{array}{c|c} \, & x_1\end{array}\right]$$
as a reduction. If $m>0$ and $M$ has a row whose right entry is $x_{i\,m+1}\neq\ast$ but for which $x_{i\,m+1}$ is not a left entry, then $$\left[\begin{array}{c|c} \ast & x_1\end{array}\right]$$ is a reduction of~$M$. If the second condition of \ref{ThmC description} is not satisfied, there exist $i,i'\in\{1,\dots,n\}$ and $j,j'\in\{1,\dots,m\}$ such that $i\neq i'$ and $x_{ij}=x_{i\,m+1}\neq\ast\neq x_{i'j'}=x_{i'\,m+1}$ but where $j\,R_{M_i} \vee R_{M_{i'}} \vee R^\ast_{M_{i,i'}} \,j'$ does not hold. By definition of $R^\ast_{M_{i,i'}}$, up to swapping $(i,j)$ with $(i',j')$, we can suppose without loss of generality that if $x_{ij''}=\ast$ or $x_{i'j''}=\ast$ for some $j''\in\{1,\dots,m\}$, then $j\,R_{M_i} \vee R_{M_{i'}} \vee R^\ast_{M_{i,i'}} \,j''$ does not hold neither. The columns of $M$ can then be divided into two disjoint sets: 
the $j''$-th column will be in the first set if $j\,R_{M_i} \vee R_{M_{i'}} \vee R^\ast_{M_{i,i'}} \,j''$ and in the second set otherwise. Interpreting, in the $i$-th and the $i'$-th rows of $M$, each entry in the first set of columns as $x_1$ and the other ones as~$\ast$, this guarantees that $M$ admits
$$\left[\begin{array}{cc|c} x_1 & \ast & x_1\\ x_1 & \ast & \ast\end{array}\right]$$
as a reduction. Suppose now the third condition of \ref{ThmC description} is not satisfied with $i,i'\in\{1,\dots,n\}$ and $j\in\{1,\dots,m\}$ such that $x_{ij}=x_{i\,m+1}\neq\ast=x_{i'\,m+1}$. If there is no $\ast$ in the left part of the $i$-th and $i'$-th rows of~$M$, it admits
$$\left[\begin{array}{c|c} x_1 & x_1\\ x_1 & \ast\end{array}\right]$$
as a reduction. Otherwise, there exists $j'\in\{1,\dots,m\}$ such that either $x_{ij'}=\ast$ or $x_{i'j'}=\ast$. By assumption, we know that $j\,R_{M_i} \vee R_{M_{i'}} \vee R^\ast_{M_{i,i'}} \,j'$ does not hold. Dividing the columns of $M$ in two sets and considering the interpretation as above, we also get that $M$ admits
$$\left[\begin{array}{cc|c} x_1 & \ast & x_1\\ x_1 & \ast & \ast\end{array}\right]$$
as a reduction.

Next, we prove \ref{ThmC description}$\Rightarrow$\ref{ThmC functionality all}. Suppose \ref{ThmC functionality all} does not hold. We can thus find a pointed set $(Y,\ast)$, $i,i'\in\{1,\dots,n\}$ and two pointed functions $f_i,f_{i'}\colon\{\ast,x_1,\dots,x_k\}\to Y$ such that $f_i(x_{ij})=f_{i'}(x_{i'j})$ for all $j\in\{1,\dots,m\}$ but $f_i(x_{i\,m+1})\neq f_{i'}(x_{i'\,m+1})$. If $i=i'$, then the first condition of~\ref{ThmC description} gets violated. Let us thus suppose that $i\neq i'$. We can see that, if $j_1,j_2\in\{1,\dots,m\}$ satisfy $j_1\, R_{M_i} \vee R_{M_{i'}} \vee R^\ast_{M_{i,i'}} \,j_2$, then $f_i(x_{ij_1})=f_i(x_{ij_2})=f_{i'}(x_{i'j_1})=f_{i'}(x_{i'j_2})$. Since we cannot have $x_{i\,m+1}=x_{i'\,m+1}=\ast$, we can also suppose without loss of generality that $x_{i\,m+1}\neq\ast$. Not to violate the first condition of~\ref{ThmC description}, there should exist $j\in\{1,\dots,m\}$ such that $x_{ij}=x_{i\,m+1}$. If $x_{i'\,m+1}\neq\ast$, then, again not to contradict the first condition of~\ref{ThmC description}, there exists $j'\in\{1,\dots,m\}$ such that $x_{i'j'}=x_{i'\,m+1}$. But since $f_i(x_{ij})=f_i(x_{i\,m+1})\neq f_{i'}(x_{i'\,m+1})=f_{i'}(x_{i'j'})$, one cannot have $j\,R_{M_i} \vee R_{M_{i'}} \vee R^\ast_{M_{i,i'}} \,j'$, violating the second condition of~\ref{ThmC description}. If $x_{i'\,m+1}=\ast$, not to contradict the third condition of~\ref{ThmC description}, there exists $j'\in\{1,\dots,m\}$ such that $j\,R_{M_i} \vee R_{M_{i'}} \vee R^\ast_{M_{i,i'}} \,j'$ and either $x_{ij'}=\ast$ or $x_{i'j'}=\ast$. By our remark above, this means that $f_i(x_{i\,m+1})=f_i(x_{ij})=f_i(x_{ij'})=f_{i'}(x_{i'j'})=\ast=f_{i'}(x_{i'\,m+1})$, leading to a contradiction. We have therefore already proved the equivalences \ref{ThmC description}$\Leftrightarrow$\ref{ThmC reductions}$\Leftrightarrow$\ref{ThmC functionality all}$\Leftrightarrow$\ref{ThmC functionality 2 elements}$\Leftrightarrow$\ref{ThmC functionality at least 2 elements}.

Since $\mathbf{Set}_\ast^\mathsf{op}$ is not equivalent to the single morphism category, the implication \ref{ThmC Set op M-closed relations}$\Rightarrow$\ref{ThmC M not trivial} is straightforward from the definition of a trivial matrix. To prove \ref{ThmC functionality all}$\Rightarrow$\ref{ThmC Set op M-closed relations}, it suffices to recall that $M$ is functional in a pointed set $(Y,\ast)$ exactly when $(Y,\ast)$ is an $\{M\}$-injective object in $\mathbf{Set}_\ast$, i.e., when internal $n$-ary relations in $\mathbf{Set}_\ast^\mathsf{op}$ are $M$-closed over $(Y,\ast)$. It remains now to prove \ref{ThmC M not trivial}$\Rightarrow$\ref{ThmC reductions}. We suppose that \ref{ThmC reductions} does not hold and we consider a finitely complete pointed category $\mathbb{C}$ with $M$-closed relations. If $M$ has $\left[\begin{array}{c|c} \, & x_1\end{array}\right]$ or $\left[\begin{array}{c|c} \ast & x_1\end{array}\right]$ as a reduction, then, for an arbitrary object $C$ in~$\mathbb{C}$, the identity morphism $C\to C$ must factors through the subobject $0\rightarrowtail C$ where $0$ is the zero object of~$\mathbb{C}$, proving that $C$ is itself a zero object. If $M$ admits $$\left[\begin{array}{c|c} x_1 & x_1\\ x_1 & \ast\end{array}\right]\quad \text{or}\quad \left[\begin{array}{cc|c} x_1 & \ast & x_1\\ x_1 & \ast & \ast\end{array}\right]$$ as a reduction, then the morphism $(1_C,0)\colon C\to C\times C$ induced by the identity on an object $C$ and the zero morphism $C\to C$ should factors through the diagonal $(1_C,1_C)\colon C\rightarrowtail C\times C$ proving again that $1_C$ is the zero morphism and thus $C$ is a zero object. This proves that $\mathbb{C}$ is equivalent to~$\mathbf{1}$ and thus $M$ is a trivial matrix.
\end{proof}

The equivalence \ref{ThmC M not trivial}$\Leftrightarrow$\ref{ThmC description} of this theorem, in the case where $M$ is a non-pointed matrix, reduces to the characterization of trivial matrices in the sense of~\cite{HoefnagelJacqminJanelidze} (i.e., non-pointed matrices $M$ for which any finitely complete category in $\mathsf{mclex}\{M\}$ is a preorder). One thus has the following corollary.

\begin{corollary}\label{corollary trivial matrices}
A non-pointed matrix $M\in\Mnp$ is trivial (in the sense of the present paper) if and only if it is trivial in the sense of~\cite{HoefnagelJacqminJanelidze}.
\end{corollary}

Let us now turn our attention to anti-trivial matrices. A matrix $M\in\M$ is said to be \emph{anti-trivial} if all finitely complete pointed categories have $M$-closed relations. The following theorem characterizes anti-trivial matrices. We invite the reader to compare the equivalence \ref{theorem anti-trivial 1}$\Leftrightarrow$\ref{theorem anti-trivial pointed sets} of it with the equivalence \ref{ThmC M not trivial}$\Leftrightarrow$\ref{ThmC Set op M-closed relations} of Theorem~\ref{ThmC}.

\begin{theorem}\label{theorem anti-trivial}
The following conditions on a matrix $M\in\M$ are equivalent:
\begin{enumerate}[label=(\arabic*)]
\item\label{theorem anti-trivial 1} $M$ is anti-trivial.
\item\label{theorem anti-trivial pointed sets} $\mathbf{Set}_\ast$ has $M$-closed relations.
\item\label{theorem anti-trivial right column} All the entries of the right column of $M$ are $\ast$'s or this right column of $M$ can be found among its left columns.
\end{enumerate}
\end{theorem}

\begin{proof}
The implication \ref{theorem anti-trivial 1}$\Rightarrow$\ref{theorem anti-trivial pointed sets} follows immediately from the definition of anti-trivial matrices. For the implication \ref{theorem anti-trivial pointed sets}$\Rightarrow$\ref{theorem anti-trivial right column}, if $M\in\M(n,m,k)$, we consider the $n$-ary pointed relation $R$ on the pointed set $\{\ast,x_1,\dots,x_k\}$ formed by the left columns of~$M$ and the column made of $n$ many~$\ast$'s. If $\mathbf{Set}_\ast$ has $M$-closed relations, $R$ should also contain the right column of~$M$, which implies~\ref{theorem anti-trivial right column}. Finally, the implication \ref{theorem anti-trivial right column}$\Rightarrow$\ref{theorem anti-trivial 1} follows from the definition of $M$-closedness.
\end{proof}

Again, the equivalence \ref{theorem anti-trivial 1}$\Leftrightarrow$\ref{theorem anti-trivial right column} of the above theorem, in the case where $M$ is a non-pointed matrix, reduces to the characterization of anti-trivial matrices in the sense of~\cite{HoefnagelJacqminJanelidze} (i.e., non-pointed matrices $M$ for which any finitely complete category is in $\mathsf{mclex}\{M\}$). One thus has the following corollary.

\begin{corollary}\label{corollary anti-trivial matrices}
A non-pointed matrix $M\in\Mnp$ is anti-trivial (in the sense of the present paper) if and only if it is anti-trivial in the sense of~\cite{HoefnagelJacqminJanelidze}.
\end{corollary}

\section{The posets $\mathsf{Mclex}$ and $\mathsf{Mclex}_\ast$ and Bourn localizations}\label{section the posets and bourn localizations}

In~\cite{HoefnagelJacqminJanelidze}, we studied the collection $\mathsf{Mclex}$ of collections $\mathsf{mclex}\{M\}$ of finitely complete categories given by non-pointed matrices~$M\in\Mnp$. In a similar way, we introduce here the collection $\mathsf{Mclex}_\ast$ of collections $\mathsf{mclex}_\ast\{M\}$ for matrices $M\in\M$. We regard $\mathsf{Mclex}$ and $\mathsf{Mclex}_\ast$ as posets with the order given by inclusion. We have an order-preserving function
$$(-)_\ast\colon\mathsf{Mclex}\to\mathsf{Mclex}_\ast$$
given by intersection with the collection of pointed categories, i.e., for a non-pointed matrix~$M$, $(-)_\ast$ sends $\mathsf{mclex}\{M\}$ to $\mathsf{mclex}_\ast\{M\}$.

In order to describe a right adjoint to this function (seen as a functor between preorder categories), we need to recall the concept of Bourn localizations. Let $\mathbb{C}$ be a given finitely complete category. A \emph{point} in $\mathbb{C}$ is a pair of morphisms $(p \colon A \twoheadrightarrow I, s \colon I \rightarrowtail A)$ such that the composite $ps$ is~$1_I$, the identity on~$I$. A morphism of points $(p,s) \rightarrow (p',s')$ is a pair of morphisms $(u \colon A \rightarrow A', v \colon I \rightarrow I')$ in $\mathbb{C}$ such that $vp=p'u$ and $us=s'v$.
$$\xymatrix{A \ar[r]^-{u} \ar@<-4pt>@{->>}[d]_-{p} & A' \ar@<-4pt>@{->>}[d]_-{p'} \\ I \ar[r]_-{v} \ar@<-4pt>@{ >->}[u]_-{s} & I' \ar@<-4pt>@{ >->}[u]_-{s'}}$$
This forms the category $\mathsf{Pt}(\mathbb{C})$ of points of~$\mathbb{C}$. Together with this category, we have a forgetful functor
\begin{align*}
\pi \colon\quad\,\mathsf{Pt}(\mathbb{C}) &\longrightarrow \mathbb{C}\\
\xymatrix{A \ar@{->>}@<-4pt>[r]_-{p} & I \ar@<-4pt>@{ >->}[l]_-{s}} &\longmapsto I\\
(u,v) & \longmapsto v
\end{align*}
called the \emph{fibration of points}~\cite{Bourn1991}. Given an object $I$ in~$\mathbb{C}$, the fibre of $\pi$ over $I$ is denoted by $\mathsf{Pt}_I(\mathbb{C})$ and consists of the subcategory of $\mathsf{Pt}(\mathbb{C})$ of points and morphisms mapped by $\pi$ to $I$ and $1_I$ respectively. Each such fibre is a finitely complete pointed category, the zero object being the point $(1_I,1_I)$. Using the terminology recalled in Example~\ref{example matrices}, D.~Bourn showed in~\cite{Bourn1996} that the following statements are equivalent:
\begin{itemize}
\item $\mathbb{C}$ is a Mal'tsev category;
\item for all objects $I$ of~$\mathbb{C}$, the fibre $\mathsf{Pt}_I(\mathbb{C})$ is unital;
\item for all objects $I$ of~$\mathbb{C}$, the fibre $\mathsf{Pt}_I(\mathbb{C})$ is strongly unital;
\item for all objects $I$ of~$\mathbb{C}$, the fibre $\mathsf{Pt}_I(\mathbb{C})$ is a Mal'tsev category.
\end{itemize}
In~\cite{Janelidze2005}, Z.~Janelidze showed that these conditions are also equivalent to
\begin{itemize}
\item for all objects $I$ of~$\mathbb{C}$, the fibre $\mathsf{Pt}_I(\mathbb{C})$ is subtractive.
\end{itemize}
As in~\cite{ZJanelidze2006b}, for a collection $\mathcal{C}$ of finitely complete pointed categories, we denote by $\mathsf{Loc}(\mathcal{C})$ the \emph{Bourn localization} of~$\mathcal{C}$, i.e., the collection of finitely complete categories $\mathbb{C}$ for which, given any object $I$ in~$\mathbb{C}$, the category $\mathsf{Pt}_I(\mathbb{C})$ is in~$\mathcal{C}$. Given a non-pointed matrix $M=[x_{ij}]_{i,j}\in\Mnp(n,m,k)$ and a variable $x$ among $x_1,x_2,x_3,\dots$, we denote by $M[x\to\ast]$ the matrix in $\M(n,m,k)$ obtained by replacing each occurrence of $x$ in $M$ by~$\ast$. We say that the pair $(M,x)$ is \emph{admissible} if there exists a left column of~$M$
$$\left[\begin{array}{c} x_{1j} \\ \vdots \\ x_{nj} \end{array}\right]$$
such that, for each $i\in\{1,\dots,n\}$, if there is a $j'\in\{1,\dots,m+1\}$ such that $x_{ij'}=x$, then $x_{ij}=x$. The following theorem generalizes the above mentioned results from~\cite{Bourn1996} and~\cite{Janelidze2005}.

\begin{theorem}\label{theorem localization}\cite{ZJanelidze2006b}.
Given a non-pointed matrix $M\in\Mnp$ and a variable $x$ among $x_1,x_2,x_3,\dots$ for which the pair $(M,x)$ is admissible, one has
$$\mathsf{mclex}\{M\}=\mathsf{Loc}(\mathsf{mclex}_\ast\{M[x\to\ast]\}).$$
\end{theorem}

In particular, this theorem shows that for a non-pointed matrix $M\in\Mnp(n,m,k)$, one has $$\mathsf{mclex}\{M\}=\mathsf{Loc}(\mathsf{mclex}_\ast\{M\})$$ provided that $m>0$. If $m=0$, $\mathsf{mclex}\{M\}$ is the collection of categories equivalent to $\mathbf{1}$ and $\mathsf{Loc}(\mathsf{mclex}_\ast\{M\})$ is the collection of finitely complete preorders. So, in any case, one has at least the inclusion $\mathsf{mclex}\{M\} \subseteq \mathsf{Loc}(\mathsf{mclex}_\ast\{M\})$. Moreover, given a (pointed) matrix $N\in\M(n,m,k)$, Theorem~\ref{theorem localization} implies that $$\mathsf{Loc}(\mathsf{mclex}_\ast\{N\})=\mathsf{mclex}\{N_{\mathsf{loc}}\}$$ where $N_{\mathsf{loc}}\in\Mnp(n,m+1,k+1)$ is obtained from $N$ by first adding a left column of $\ast$'s and then replace all $\ast$'s by~$x_{k+1}$. Bourn localizations thus induce an order-preserving function
$$\mathsf{Loc}\colon\mathsf{Mclex}_\ast\to\mathsf{Mclex}$$
sending $\mathsf{mclex}_\ast\{N\}$ to $\mathsf{Loc}(\mathsf{mclex}_\ast\{N\})$. It is easy to see that this function gives a right adjoint to $(-)_\ast\colon\mathsf{Mclex}\to\mathsf{Mclex}_\ast$, i.e., that
$$\mathsf{mclex}\{M\} \subseteq \mathsf{Loc}(\mathsf{mclex}_\ast\{N\}) \Longleftrightarrow \mathsf{mclex}_\ast\{M\} \subseteq \mathsf{mclex}_\ast\{N\}$$
holds for any non-pointed matrix $M$ and any (pointed) matrix~$N$. Indeed, supposing $\mathsf{mclex}\{M\} \subseteq \mathsf{Loc}(\mathsf{mclex}_\ast\{N\})$ and given a finitely complete pointed category $\mathbb{C}$ with $M$-closed relations, we know that all fibres $\mathsf{Pt}_I(\mathbb{C})$ have $N$-closed relations. Taking $I$ to be a zero object of $\mathbb{C}$ gives a category which is isomorphic to $\mathbb{C}$ which thus has $N$-closed relations. The converse inclusion follows immediately from $\mathsf{mclex}\{M\} \subseteq \mathsf{Loc}(\mathsf{mclex}_\ast\{M\})$.
$$\xymatrix{\mathsf{Mclex} \ar@<-4pt>@/_/[rr]_-{(-)_\ast}="A" && \mathsf{Mclex}_\ast \ar@<-4pt>@/_/[ll]_-{\mathsf{Loc}}="B" \ar@{} "A";"B"|-{\top}}$$
In addition, this right adjoint is also `almost a left inverse' of $(-)_\ast$ in the sense that $\mathsf{mclex}\{M\}=\mathsf{Loc}(\mathsf{mclex}_\ast\{M\})$ holds for any non-pointed matrix $M$ with at least one left column. Therefore, if $M$ and $M'$ are two non-pointed matrices such that $\mathsf{mclex}_\ast\{M\} \subseteq \mathsf{mclex}_\ast\{M'\}$, then $\mathsf{mclex}\{M\} \subseteq \mathsf{Loc}(\mathsf{mclex}_\ast\{M'\})$ and thus either $\mathsf{mclex}\{M\}\subseteq\mathsf{mclex}\{M'\}$ or $\mathsf{mclex}\{M\}$ is the collection of all finitely complete preorders and $\mathsf{mclex}\{M'\}$ is the collection of all categories equivalent to~$\mathbf{1}$. This discussion leads us to the following theorem.

\begin{theorem}\label{theorem restriction is almost injective}
Let $S$ and $U$ be non-pointed matrix sets such that $U$ does not contain any matrix with no left columns. Then one has $\mathsf{mclex}_\ast S\subseteq \mathsf{mclex}_\ast U$ if and only if $\mathsf{mclex} S\subseteq \mathsf{mclex} U$.
\end{theorem}

Let us immediately note that, if $S$ is the singleton consisting of the trivial matrix $\left[\begin{array}{c|c} x_2 & x_1\end{array}\right]\in\Mnp(1,1,2)$ and $U$ is the singleton consisting of the trivial matrix $\left[\begin{array}{c|c} \, & x_1\end{array}\right]\in\Mnp(1,0,1)$, one has that $\mathsf{mclex} U=\mathsf{mclex}_\ast S=\mathsf{mclex}_\ast U$ is the collection of all categories equivalent to~$\mathbf{1}$ while $\mathsf{mclex} S$ is the collection of all finitely complete preorders. Therefore, $\mathsf{mclex}_\ast S\subseteq \mathsf{mclex}_\ast U$ holds but $\mathsf{mclex} S\subseteq \mathsf{mclex} U$ does not. In other words, the function $(-)_\ast$ is not injective but is not far from being it, i.e., for two non-pointed matrices $M$ and~$M'$, one has
$$\mathsf{mclex}_\ast\{M\} = \mathsf{mclex}_\ast\{M'\} \Longleftrightarrow (\mathsf{mclex}\{M\}=\mathsf{mclex}\{M'\} \text{ or } (M\text{ and }M'\text{ are trivial matrices})).$$

We know from~\cite{HoefnagelJacqminJanelidze} that $\mathsf{Mclex}$ is infinite but countable. In view of the above properties of the function $(-)_\ast$, we can deduce that $\mathsf{Mclex}_\ast$ is also infinite and it is not hard to see that it is also countable. On the other hand, it is proved in~\cite{HoefnagelJacqminJanelidze} that $\mathsf{Mclex}$ is a meet semi-lattice with the meet operation being given by the intersection of the corresponding collections of categories. It seems the same argument cannot be transposed to the pointed context and we therefore leave as an open question whether finite conjunctions of matrix properties of finitely complete pointed categories are again matrix properties, and the question whether $\mathsf{Mclex}_\ast$ is a meet semi-lattice.

Given integers $n>0$ and $m,k\geqslant 0$, we have denoted in~\cite{HoefnagelJacqminJanelidze} by $\mathsf{Mclex}[n,m,k]$ the restriction of the poset $\mathsf{Mclex}$ to those $\mathsf{mclex}\{M\}$ given by non-pointed matrices $M\in\Mnp(n,m,k)$. Adapting this notation, we denote by $\mathsf{Mclex}_\ast[n,m,k]$ the restriction of $\mathsf{Mclex}_\ast$ to those matrix classes $\mathsf{mclex}_\ast\{M\}$ given by matrices $M\in\M(n,m,k)$. Obviously, each of these posets is finite. For very small values of the parameters $n,m,k$, in view of Theorems~\ref{ThmC} and~\ref{theorem anti-trivial}, we can have a complete description of these posets:
\begin{itemize}
\item $\mathsf{Mclex}_\ast[n,m,0]$ contains only the collection of all finitely complete pointed categories.
\item When $k>0$, $\mathsf{Mclex}_\ast[n,0,k]=\mathsf{Mclex}_\ast[n,1,k]=\mathsf{Mclex}_\ast[1,m,k]$ has exactly two matrix classes given by the collection of all finitely complete pointed categories and the collection of categories equivalent to~$\mathbf{1}$.
\end{itemize}

In view of this, we will call a matrix class (of finitely complete pointed categories) \emph{degenerate} if it is determined by a trivial matrix or an anti-trivial matrix. That is, the degenerate matrix classes are the following two:
\begin{itemize}
\item the matrix class given by trivial matrices, i.e., the collection of categories equivalent to~$\mathbf{1}$, or in other words, the bottom element of the poset $\mathsf{Mclex}_\ast$;
\item the matrix class given by anti-trivial matrices, i.e., the collection of all finitely complete pointed categories, or in other words, the top element of the poset $\mathsf{Mclex}_\ast$.
\end{itemize}

\begin{remark}\label{remark localization classes}
Since the left adjoint $(-)_\ast$ preserves the top element, we know that the collection of all finitely complete pointed categories is the only element of $\mathsf{Mclex}_\ast$ sent by $\mathsf{Loc}$ to the top element of $\mathsf{Mclex}$, i.e., the collection of all finitely complete categories. Similarly, in view of Theorem~\ref{ThmC}, it is not hard to see that given a non-trivial matrix~$N$, the non-pointed matrix $N_{\mathsf{loc}}$ is also non-trivial and thus the bottom element of $\mathsf{Mclex}_\ast$ is the only element being mapped by $\mathsf{Loc}$ to the collection of finitely complete preorders, i.e., the unique atom of $\mathsf{Mclex}$. Let us also notice that, given a non-pointed matrix $M$ with at least one left column, in view of the counit of the adjunction $(-)_\ast\dashv\mathsf{Loc}$, the matrix class $\mathsf{mclex}_\ast\{M\}$ is the smallest element of $\mathsf{Mclex}_\ast$ whose Bourn localization is $\mathsf{mclex}\{M\}$.
\end{remark}

\begin{remark}\label{remark restriction loc}
Given integers $n>0$ and $m,k\geqslant 0$, we know that the Bourn localization of a matrix class $\mathsf{mclex}_\ast\{N\}$ for a matrix $N\in\M(n,m,k)$ is given by $\mathsf{mclex}\{N_{\mathsf{loc}}\}$, where $N_{\mathsf{loc}}\in\Mnp(n,m+1,k+1)$. The function $\mathsf{Loc}$ therefore restricts as an order-preserving function
$$\mathsf{Loc}\colon\mathsf{Mclex}_\ast[n,m,k]\twoheadrightarrow\mathsf{Mclex}[n,m+1,k+1]$$
(denoted in the same way by abuse of notation). Moreover, in view of the construction of $N_{\mathsf{loc}}$ from~$N$, it is easy to see that this function is actually surjective.
\end{remark}

Let us notice that a matrix $M\in\M(n,m,k)$ with $n>0$ and $m,k\geqslant 0$ (i.e., a $n\times (m+1)$ matrix with entries in $\{\ast,x_1,\dots,x_k\}$) can have at most $(k+1)^n$ different left columns. One of these columns is the right column of~$M$, whose presence makes $M$ anti-trivial. Among those $(k+1)^n$ possibilities for the left columns of~$M$, there is also the column of~$\ast$'s. If this column is the right column, $M$ is again anti-trivial. If the column of $\ast$'s belongs to the left of~$M$, one can remove it without changing the corresponding matrix class. Therefore, one has the inclusion
$$\mathsf{Mclex}_\ast[n,m,k]\subseteq\mathsf{Mclex}_\ast[n,(k+1)^n-2,k]$$
for all $n,k>0$ and $m\geqslant 0$. Applying Remark~\ref{remark restriction loc} here, we know that $\mathsf{Loc}$ restricts to a surjective function
$$\mathsf{Loc}\colon\mathsf{Mclex}_\ast[n,(k+1)^n-2,k]\twoheadrightarrow\mathsf{Mclex}[n,(k+1)^n-1,k+1]$$ where, as shown in~\cite{HoefnagelJacqminJanelidze}, $\mathsf{Mclex}[n,(k+1)^n-1,k+1]$ contains all elements of $\mathsf{Mclex}$, different from the bottom element, and induced by a non-pointed matrix with $n$ rows and $k+1$ variables. Moreover, if a matrix $M\in\M(n,m,k)$ has the left entries of its $i$-th row consisting of $m$ different variables, then, for any $i'\in\{1,\dots,n\}$, the equivalence relation $R_{M_i} \vee R_{M_{i'}} \vee R^\ast_{M_{i,i'}}$ used in Theorem~\ref{ThmC} is nothing but $R_{M_{i'}}$. In view of this and of that theorem, it is easy to prove that such a matrix induces a degenerate matrix class $\mathsf{mclex}_\ast\{M\}$. This proves that
$$\mathsf{Mclex}_\ast[n,m,k]\subseteq\mathsf{Mclex}_\ast[n,m,m-1]$$
for all $n>0$, $m>1$ and $k\geqslant 0$. In addition, if $M$ is a non-trivial matrix in $\M(n,m,m-1)$, up to permutation of variables, we can assume that the right column is only made of $x_1$'s and~$\ast$'s. There are $m^m$ different rows with $x_1$ as right entry, to which we can subtract the $(m-1)^m$ rows which does not contain $x_1$ on the left (in view of Theorem~\ref{ThmC}). There are also $m^m$ different rows with $\ast$ as a right entry, to which we can remove the row of $\ast$'s which does not influence the corresponding matrix class (see for instance our algorithm in the next section). The matrix class induced by $M$ is thus the same as the matrix class determined by a matrix made of a maximum of $2m^m-(m-1)^m-1$ rows. This observation leads to the fact that the inclusion
$$\mathsf{Mclex}_\ast[n,m,k]\subseteq\mathsf{Mclex}_\ast[2m^m-(m-1)^m-1,m,m-1]$$
holds for all $n>0$, $m>1$ and $k\geqslant 0$.

\section{The algorithm}\label{section the algorithm}

Theorem~\ref{theorem restriction is almost injective} shows that, to decide whether $\mathsf{mclex}_\ast S\subseteq \mathsf{mclex}_\ast U$ for finite non-pointed matrix sets $S$ and~$U$, it is equivalent, in non-trivial cases, to use the algorithm of~\cite{HoefnagelJacqminJanelidze} which decide whether $\mathsf{mclex} S\subseteq \mathsf{mclex} U$. In this section, we present an analogue algorithm for the case where $S$ and $U$ are not necessarily non-pointed. In order to do so, adapting the ideas of~\cite{HoefnagelJacqminJanelidze}, themselves based on the ideas of~\cite{Hoefnagel2019}, let us now study $S$-injective objects in the category of $n$-ary pointed relations. For each non-zero natural number~$n$, we consider the category $\mathbf{Rel}_\ast^n$ whose objects are triples $(X,\ast,R)$ where $(X,\ast)$ is a pointed set and $R$ is an $n$-ary pointed relation on $X$ (i.e., a subset $R\subseteq X^n$ containing $(\ast,\dots,\ast)$), and whose morphisms $(X,\ast,R)\to (X',\ast',R')$ are relation-preserving pointed functions, i.e., functions $f\colon X\to X'$ such that $f(\ast)=\ast'$ and such that there exists a (necessarily unique) dashed morphism rendering the diagram
$$\xymatrix{R \ar@{}[d]|(.21){}="A" \ar@{^{(}->}@<-2pt>"A";[d] \ar@{-->}[r] & R' \ar@{}[d]|(.21){}="B" \ar@{^{(}->}@<-2pt>"B";[d] \\ X^n \ar[r]_-{f^n} & X'^{\,n}}$$
commutative. The forgetful functor $\mathbf{Rel}_\ast^n\to\mathbf{Set}_\ast$ is a topological functor (see e.g.~\cite{Borceux1994b}) and therefore $\mathbf{Rel}_\ast^n$ is a complete and cocomplete pointed category. Moreover, the limit/colimit of a diagram in $\mathbf{Rel}_\ast^n$ has as its underlying pointed set the limit/colimit of the underlying diagram in $\mathbf{Set}_\ast$, equipped with the largest/smallest pointed relation making the canonical projections/inclusions relation-preserving. It is then not difficult to see that, given a matrix $M\in\M$, an object $(X,\ast,R)$ of $\mathbf{Rel}_\ast^n$ is an $\{M\}$-injective object if and only if $M$ is functional in the pointed set $(X,\ast)$ and $R$ is $M$-sharp. This is because, if we consider the diagram in Theorem~\ref{theorem description injective objects} applied to $(X,\ast,R)$ in $\mathbf{Rel}_\ast^n$, we have that $M$ functional in the pointed set $(X,\ast)$ and $R$ is $M$-sharp if and only if the pointed function $p^X_M$ exists and is a morphism in $\mathbf{Rel}_\ast^n$, i.e., is relation-preserving. Furthermore, when $M$ is not a trivial matrix, we can drop `$M$ is functional in~$(X,\ast)$' by virtue of Theorem~\ref{ThmC}, i.e., $(X,\ast,R)$ is $\{M\}$-injective if and only if $R$ is $M$-sharp. We can also see that for a matrix set~$S$, the category $\mathsf{Inj}_S\mathbf{Rel}_\ast^n$ forms a full reflective subcategory of~$\mathbf{Rel}_\ast^n$. Given an object $(X,\ast,R)$ in $\mathbf{Rel}_\ast^n$, its reflection $f\colon (X,\ast,R)\to (X',\ast',R')$ in the subcategory $\mathsf{Inj}_S\mathbf{Rel}_\ast^n$ is obtained as follows:
\begin{itemize}
\item If $S$ contains a trivial matrix, then $(X',\ast',R')$ is the zero object of $\mathbf{Rel}_\ast^n$, i.e., $X'=\{\ast'\}$ and $R'=\{\ast'\}^n$, and the function $f$ is the unique function $f\colon X\to \{\ast'\}$.

\item If $S$ does not contain a trivial matrix, then $(X',\ast')=(X,\ast)$, the function $f$ is the identity function $f=1_X$ and $R'\subseteq X^n$ is the intersection of all relations containing $R$ as a subrelation and which are $M$-sharp for each matrix $M$ in~$S$.
\end{itemize}
Thus, a colimit of a small diagram in $\mathsf{Inj}_S\mathbf{Rel}_\ast^n$ can be obtained by applying the construction above to the colimit of the same diagram in $\mathbf{Rel}_\ast^n$. In particular, each $(\mathsf{Inj}_S\mathbf{Rel}_\ast^n)^\mathsf{op}$ is a finitely complete pointed category. These remarks together with Theorem~\ref{ThmB} bring us to the following result:

\begin{theorem}\label{ThmD}
Consider two matrix sets $S$ and~$U$. If for any non-zero natural number~$n$, every $n$-ary pointed relation $R$ on any pointed set $(X,\ast)$ that is $M$-sharp for every matrix $M$ in $S$ is $N$-closed for every matrix $N$ in $U$ having $n$ rows, then every finitely complete pointed category that has $M$-closed relations for every $M$ in $S$ also has $N$-closed relations for every $N$ in~$U$, i.e., $\mathsf{mclex}_\ast S\subseteq\mathsf{mclex}_\ast U$. The converse is also true when no matrix in $S$ is trivial.  
\end{theorem}

\begin{proof}
Assume first that for any integer $n>0$, every $n$-ary pointed relation $R$ on any pointed set $(X,\ast)$ that is $M$-sharp for every matrix $M$ in $S$ is $N$-closed for every matrix $N$ in $U$ with $n$ rows. Consider a finitely complete pointed category $\mathbb{C}$ that has $M$-closed relations for every $M$ in~$S$. Then every internal relation in $\mathbb{C}$ is $M$-sharp for each $M\in S$. Considering a matrix $N\in U$ with $n$ rows, an $n$-ary internal relation $r\colon R \rightarrowtail C^n$ on an object $C$ in~$\mathbb{C}$ and an arbitrary object~$Y$, we need to prove that $r$ is $N$-closed over~$Y$. These induce an $n$-ary relation $R'$ on the pointed set $\mathbb{C}(Y,C)$ for which $g_1,\dots,g_n\colon Y\to C$ are related if and only if there exists a (necessarily unique) morphism $v\colon Y \to R$ such that $r_iv=g_i$ for each $i\in\{1,\dots,n\}$. It is then easy to see that $R'$ is an (ordinary) $n$-ary pointed relation which is $M$-sharp for each $M\in S$. Our assumption gives then that $R'$ is $N$-closed, which implies that the internal relation $r$ on $C$ is $N$-closed over~$Y$.

Now assume that every finitely complete pointed category that has $M$-closed relations for every $M$ in $S$ also has $N$-closed relations for every $N$ in~$U$. Suppose also that no matrix in $S$ is trivial. Then no matrix in $U$ can be trivial as well since, by Theorem~\ref{ThmC}, $\mathbf{Set}_\ast^{\mathsf{op}}$ belongs to $\mathsf{mclex}_\ast S$ and thus to $\mathsf{mclex}_\ast U$. For each positive~$n$, consider the full subcategory $\mathsf{Inj}_U\mathbf{Rel}_\ast^n$ of $\mathbf{Rel}_\ast^n$ consisting of $U$-injective objects. We will now apply Theorem~\ref{ThmB} to the category $\mathbb{C}=\mathbf{Rel}_\ast^n$. Let $T$ be the set of all non-trivial matrices. For any pointed set~$(X,\ast)$, the relation $X^n$ is clearly $M$-sharp for every matrix~$M$. So the object $(X,\ast,X^n)$ is a $T$-injective object in $\mathbf{Rel}_\ast^n$. For any object $(X,\ast,R)$ in~$\mathbf{Rel}_\ast^n$, there is a monomorphism into such object --- namely, the inclusion $(X,\ast,R)\rightarrowtail (X,\ast,X^n)$. By Theorem~\ref{ThmB} we then get that $\mathsf{Inj}_U\mathbf{Rel}_\ast^n$ is the largest full subcategory of $\mathbf{Rel}_\ast^n$ having the following properties: it contains all $T$-injective objects, it is closed under finite limits and its dual category has $N$-closed relations for every $N$ in~$U$. By Theorem~\ref{ThmB} again, the subcategory $\mathsf{Inj}_S\mathbf{Rel}_\ast^n$ of $\mathbf{Rel}_\ast^n$ consisting of all $S$-injective objects has similar properties: it contains all $T$-injective objects, it is closed under finite limits and its dual category has $M$-closed relations for every $M$ in~$S$. The assumption that every finitely complete pointed category having $M$-closed relations for every $M$ in $S$ has $N$-closed relations for every $N$ in $U$ gives that $\mathsf{Inj}_S\mathbf{Rel}_\ast^n$ must be a subcategory of $\mathsf{Inj}_U\mathbf{Rel}_\ast^n$ (note that, from the paragraph before the theorem, we know that the duals of these categories are finitely complete pointed). So, for any non-zero natural number~$n$, every $n$-ary pointed relation $R$ on any pointed set $(X,\ast)$ that is $M$-sharp for every matrix $M$ in $S$ is also $N$-sharp for every matrix $N$ in~$U$. In particular, this means that every such relation is $N$-closed for every matrix $N$ in $U$ having $n$ rows.
\end{proof}

From the proof of Theorem~\ref{ThmD} we can extract the following theorem, which says that in order to prove an implication between matrix properties (in the finitely complete pointed context), one only needs to produce a proof for a single particular category.

\begin{theorem}\label{ThmF}
Let $S$ be a matrix set and let $N\in\M$ be a matrix with $n$ rows. The following statements are equivalent:
\begin{enumerate}[label=(\arabic*)]
\item $\mathsf{mclex}_\ast S\subseteq\mathsf{mclex}_\ast\{N\}$.
\item The category $(\mathsf{Inj}_S\mathbf{Rel}_\ast^n)^\mathsf{op}$, which belongs to $\mathsf{mclex}_\ast S$, also belongs to $\mathsf{mclex}_\ast\{N\}$.
\end{enumerate}
\end{theorem}

(Notice that in the case where $S$ contains a trivial matrix, both statements are true and the result thus also holds in that case.)

We now come to the following question. Given a matrix set $S$ and a matrix $N\in\M(n,m,k)$, how to decide whether every $n$-ary pointed relation $R$ on any pointed set $(X,\ast)$ that is $M$-sharp for every matrix $M$ in $S$ is $N$-closed? Let us denote by $\mathsf{col}_\ast(N)$ the $n$-ary pointed relation on $\{\ast,x_1,\dots,x_k\}$ containing exactly the left columns of $N$ and $(\ast,\dots,\ast)$. Let us denote by $\mathsf{col}_\ast^S(N)$ the intersection of all $n$-ary relations on $\{\ast,x_1,\dots,x_k\}$ that contain $\mathsf{col}_\ast(N)$ as a subrelation and that are $M$-sharp for all $M$ in~$S$. It is certainly pointed and $M$-sharp for all $M$ in~$S$. Therefore, if every $n$-ary pointed relation on any pointed set $(X,\ast)$ that is $M$-sharp for every matrix $M$ in $S$ is $N$-closed, then the right column $N_\mathsf{r}$ of $N$ belongs to $\mathsf{col}_\ast^S(N)$. The converse is also true: If $N_\mathsf{r}\in \mathsf{col}_\ast^S(N)$, then every $n$-ary pointed relation $R$ on any pointed set $(X,\ast)$ that is $M$-sharp for every matrix $M$ in $S$ is $N$-closed. To see this, consider an interpretation 
$$\left[\begin{array}{ccc|c} b_{11} & \dots & b_{1m} & b_{1\,m+1}\\ \vdots & & \vdots & \vdots\\ b_{n1} & \dots & b_{nm} & b_{n\,m+1} \end{array}\right]$$
of $N$ of type $(X,\ast)$ such that every left column of the interpretation belongs to~$R$. Let this interpretation be given by a pointed map $f\colon \{\ast,x_1,\dots,x_k\}\to X$. Then the inverse image $f^{-1}(R)$ of $R$ along $f$ will be a pointed relation on $\{\ast,x_1,\dots,x_k\}$ that is $M$-sharp for every $M$ in~$S$. Indeed, every reduction of each $M$ of type $(\{\ast,x_1,\dots,x_k\},\dots,\{\ast,x_1,\dots,x_k\})$ whose left columns belong to $f^{-1}(R)$ will have a further reduction by applying $f$ to its entries, whose left columns belong to~$R$. Then, since $R$ is $M$-sharp, $f$ of the right column of the reduction will be in~$R$, which means the right column of the reduction will be in $f^{-1}(R)$. Now, since $f^{-1}(R)$ is $M$-sharp for every $M$ in $S$ and since it contains $\mathsf{col}_\ast(N)$, it must also contain $\mathsf{col}_\ast^S(N)$. The fact that $N_\mathsf{r}\in \mathsf{col}_\ast^S(N)$ will now give that the right column of the above interpretation of $N$ belongs to~$R$. So Theorem~\ref{ThmD} has the following consequence.

\begin{corollary}\label{CorA}  
Consider two matrix sets $S$ and~$U$. If $N_\mathsf{r}\in\mathsf{col}_\ast^S(N)$ for every matrix $N$ in~$U$, then $\mathsf{mclex}_\ast S\subseteq\mathsf{mclex}_\ast U$. The converse is also true when no matrix in $S$ is trivial.  
\end{corollary}

Note that $\mathsf{col}_\ast^S(N)$ is necessarily finite (it is a subset of $\{\ast,x_1,\dots,x_k\}^n$, when $N\in\M(n,m,k)$). When $S$ is finite, we can build $\mathsf{col}_\ast^S(N)$ in finitely many steps as follows. For a given pointed relation $R\subseteq \{\ast,x_1,\dots,x_k\}^n$, we consider the pointed relation $S(R)\subseteq \{\ast,x_1,\dots,x_k\}^n$ containing exactly $R$ and the right columns of each row-wise interpretation $B$ of type $(\{\ast,x_1,\dots,x_k\},\dots,\{\ast,x_1,\dots,x_k\})$ of each matrix $M'\in\M(n,m',k')$ whose rows are rows of a common matrix $M$ in $S$ such that the left columns of $B$ belong to~$R$. The number of all possible such $B$'s is finite (as $S$ is finite), so $S(R)$ can be built from $R$ in finitely many steps. It is easy to see that if $R\subseteq \mathsf{col}_\ast^S(N)$ then also $S(R)\subseteq \mathsf{col}_\ast^S(N)$. Now, starting with $R=\mathsf{col}_\ast(N)$, we can build a chain of proper subset inclusions
$$\mathsf{col}_\ast(N)\subset S(\mathsf{col}_\ast(N))\subset SS(\mathsf{col_\ast}(N))\subset\dots\subset SS\dots S(\mathsf{col_\ast}(N))$$
until $S(R)$ for the last set $R$ in the chain is equal to~$R$. When this happens, $R$ will be $M$-sharp for every $M$ in~$S$, and hence $R=\mathsf{col}_\ast^S(N)$. The chain does terminate after some finitely many steps because each of the sets in the chain are subsets of the finite set $\{\ast,x_1,\dots,x_k\}^n$.

From the results above we can readily extract an algorithm for deciding $\mathsf{mclex}_\ast S\subseteq \mathsf{mclex}_\ast U$, when $S$ and $U$ are finite matrix sets. Note that the results did not give full characterization of $\mathsf{mclex}_\ast S\subseteq \mathsf{mclex}_\ast U$ when $S$ contains a trivial matrix, however, when $S$ contains a trivial matrix, one always has $\mathsf{mclex}_\ast S\subseteq \mathsf{mclex}_\ast U$.

The combined algorithm which deals with both the trivial and non-trivial matrices is then the following. To decide whether $\mathsf{mclex}_\ast S\subseteq \mathsf{mclex}_\ast U$ where $S$ and $U$ are finite matrix sets, do the following:
\begin{description}
	\item[Step 1.] If $S$ contains a trivial matrix (i.e., if $S$ contains a matrix which does not satisfies the three conditions in~\ref{ThmC description} of Theorem~\ref{ThmC}), then terminate the process with positive decision for $\mathsf{mclex}_\ast S\subseteq \mathsf{mclex}_\ast U$.

    \item[Step 2.] For each matrix $N\in\M(n,m,k)$ in $U$ do the following. First expand the left part of $N$ with a column of $\ast$'s. Then, keep expanding the left part of~$N$, until it is impossible to expand it further, with right columns of those row-wise interpretations $B$ of type $(\{\ast,x_1,\dots,x_k\},\dots,\{\ast,x_1,\dots,x_k\})$ of each matrix $M'\in\M(n,m',k')$ whose rows are rows of a common matrix $M$ in $S$ such that the left columns of $B$, but not its right column, can be found in the left part of~$N$. If the expanded left part of $N$ does not contain the right column of~$N$, then terminate the process with negative decision for $\mathsf{mclex}_\ast S\subseteq \mathsf{mclex}_\ast U$.

    \item[Step 3.] Reaching this step means that the process has not been terminated in the previous steps. Then the process completes with positive decision for $\mathsf{mclex}_\ast S\subseteq \mathsf{mclex}_\ast U$.
\end{description}

We conclude this section by recalling from~\cite{HoefnagelJacqminJanelidze} (and Theorem~\ref{theorem restriction is almost injective}) that, given two matrices $M$ and~$N$, the above algorithm \emph{cannot} be used to prove that any pointed variety of universal algebras with $M$-closed relations has $N$-closed relations. In general, this statement is weaker than the statement $\mathsf{mclex}_\ast\{M\}\subseteq \mathsf{mclex}_\ast\{N\}$. Actually, the statement $\mathsf{mclex}_\ast\{M\}\subseteq \mathsf{mclex}_\ast\{N\}$ is, in general, even stronger than the statement that any regular~\cite{BGV1971} well-powered pointed category with $M$-closed relations has $N$-closed relations. However, it is shown in~\cite{HJ} that if $M$ is non-pointed and $N$ is the matrix from Example~\ref{example matrices} determining Mal'tsev categories, this distinction disappears. We refer the reader to~\cite{HJJW} for a link of this case with the matrix class of majority categories.

\section{Computer-assisted results}\label{section computer-assisted results}

A computer implementation of our algorithm allows us to get a complete description of the posets $\mathsf{Mclex}_\ast[n,m,k]$ for relatively small values of $n,m,k$. We present those results in this section.

Before displaying these results, we describe how we chose to represent the matrix classes in the display. For each matrix class $\mathcal{C}\in\mathsf{Mclex}_\ast$, and for each choice of integers $n>0$ and $m,k\geqslant 0$, consider the set $\mathcal{C}_{n,m,k}$ of all matrices $M$ such that $\mathcal{C}=\mathsf{mclex}_\ast\{M\}$ and $M\in \M(n',m',k')$ for some $n'\leqslant n$, $m'\leqslant m$ and $k'\leqslant k$. Now consider the subset $\mathcal{C}^{\mathsf{R}}_{n,m,k}$ of $\mathcal{C}_{n,m,k}$ consisting of those matrices that have minimal number of rows; the subset $\mathcal{C}^{\mathsf{RC}}_{n,m,k}$ of $\mathcal{C}^{\mathsf{R}}_{n,m,k}$ consisting of those matrices that have minimal number of columns and the subset $\mathcal{C}^{\mathsf{RCV}}_{n,m,k}$ of $\mathcal{C}^{\mathsf{RC}}_{n,m,k}$ consisting of those matrices whose entries lie in $\{\ast,x_1,\dots,x_v\}$ for the smallest possible~$v$. We then view each matrix in $\mathcal{C}^{\mathsf{RCV}}_{n,m,k}$ as a sequence of elements from $\{x_1,\dots,x_k,\ast\}$ by juxtaposing the transpose of each column next to each other, starting with the right column and continuing with the left columns from left to right. We order these matrices using the lexicographical ordering coming from this way of displaying matrices and from the order $x_1<\dots<x_k<\ast$. A matrix will be called an \emph{$(n,m,k)$-canonical} matrix if it is the smallest element of $\mathcal{C}^{\mathsf{RCV}}_{n,m,k}$ under this ordering for some $\mathcal{C}\in\mathsf{Mclex}_\ast$. It is easy to see that any matrix $M\in\M(n',m',k')$ which is $(n,m,k)$-canonical for some $n\geqslant n'$, $m \geqslant m'$ and $k \geqslant k'$ is also $(n',m',k')$-canonical. In this section, we will however see that, for the matrix class of strongly unital categories (see Example~\ref{example matrices}), the $(2,3,2)$-canonical matrix is not the same as the $(3,3,1)$-canonical matrix. In the following lemma, we again consider that the set of possible entries of a matrix is ordered via $x_1<\dots<x_k<\ast$.

\begin{lemma}
Let $n>0$ and $m,k\geqslant 0$ be integers and $M$ an $(n,m,k)$-canonical matrix. The following properties hold for~$M$:
\begin{itemize}
\item There exist integers $a,b\geqslant 0$ such that the right column of $M$ consists of $a$ many $x_1$'s followed by $b$ many~$\ast$'s.
\item $M$ has no duplicate rows or left columns.
\item Each left column of $M$ contains at least one variable (i.e., an entry different from~$\ast$).
\item The left columns of $M$ are lexicographically ordered.
\item The first $a$ rows of $M$ are lexicographically ordered.
\item The last $b$ rows of $M$ are lexicographically ordered.
\item For each of the first $a$ rows of $M$ and for each $1<i<j\leqslant k$, if $x_j$ appears in the left part of the row, then so does $x_i$ and the first occurrence of $x_j$ appears after the first occurrence of~$x_i$.
\item For each of the last $b$ rows of $M$ and for each $1\leqslant i<j\leqslant k$, if $x_j$ appears in the left part of the row, then so does $x_i$ and the first occurrence of $x_j$ appears after the first occurrence of~$x_i$.
\end{itemize}
\end{lemma}

The proof of this lemma is similar to the proof of Lemma~4.1 in~\cite{HoefnagelJacqminJanelidze}.

\begin{proof}
This follows easily from the fact that, given a matrix in some $\M(n',m',k')$, one can, without changing the corresponding matrix property, do the following:
\begin{itemize}
\item in each row, permute the variables $x_1,\dots,x_{k'}$;
\item permute the rows and the left columns;
\item delete the duplicated rows and left columns;
\item delete a left column which contains only~$\ast$'s.\qedhere
\end{itemize}
\end{proof}

We now show the computations of the posets $\mathsf{Mclex}_\ast[n,m,k]$ for various values of the parameters $n,m,k$ obtained by the computer. When displaying such a poset $\mathsf{Mclex}_\ast[n,m,k]$, we represent each matrix class by the $(n,m,k)$-canonical matrix giving rise to the considered matrix class. This matrix will be represented by a gird where the right column of the matrix is highlighted in gray and where each variable $x_i$ has been represented by its index~$i$. An arrow from a matrix $M$ to a matrix $N$ represents the inclusion $\mathsf{mclex}_\ast\{M\}\subseteq\mathsf{mclex}_\ast\{N\}$. We do not draw an arrow if it can be obtained by a path of arrows (we thus display the reflexive and transitive reduction of the poset).

We start with Figure~\ref{figure 2,3,2} which represents the poset $\mathsf{Mclex}_\ast[2,3,2]$. According to our convention, the bottom element of that poset is the matrix class generated by the trivial matrix $\left[\begin{array}{c|c} \, & x_1\end{array}\right]$, i.e., the collection of categories equivalent to~$\mathbf{1}$. The top element is the matrix class generated by the anti-trivial matrix $\left[\begin{array}{c|c} \, & \ast\end{array}\right]$, i.e., the collection of all finitely complete pointed categories. In view of Example~\ref{example matrices}, the matrix on the left of the second row (from the top) represents the collection of subtractive categories and the matrix on the right of the second row represent the collection of unital categories. The matrix in the third row represents the collection of strongly unital categories and the matrix in the fourth row represents the collection of pointed Mal'tsev categories. This poset already illustrates a result from~\cite{Janelidze2005}, where it is shown that a finitely complete pointed category is strongly unital if and only if it is subtractive and unital. This situation is somehow special as we recall that in general we do not know whether the intersection of two matrix classes is again a matrix class. It also illustrates that pointed Mal'tsev categories are in particular strongly unital, which was shown in~\cite{Bourn1996}. We also computed the posets $\mathsf{Mclex}_\ast[2,14,3]$ and $\mathsf{Mclex}_\ast[2,10,4]$ which are both equal to $\mathsf{Mclex}_\ast[2,3,2]$. We therefore conjecture that Figure~\ref{figure 2,3,2} shows all matrix classes induced by two-row matrices, although we have not been able to formally prove it yet.
\begin{figure}[htb!]
\centering
    \includegraphics[width=143pt]{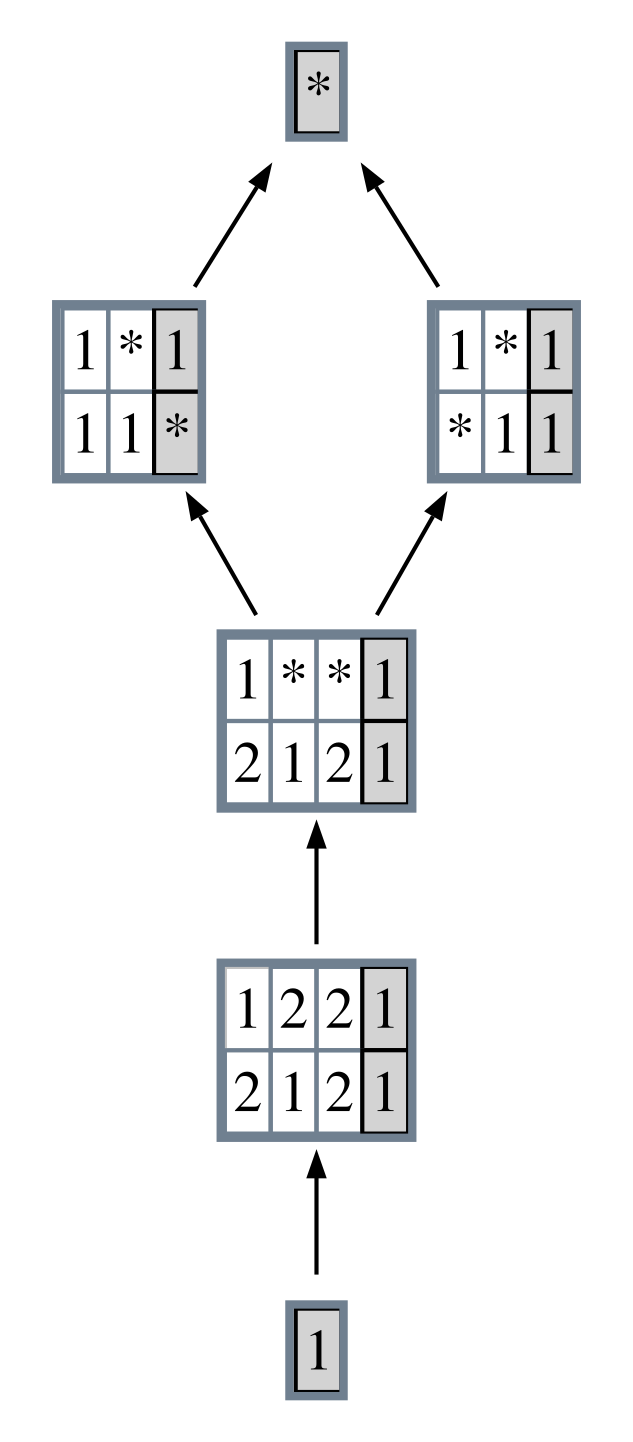}
    \caption{Hasse diagram of the poset $\mathsf{Mclex}_\ast[2,3,2]$.}
    \label{figure 2,3,2}
\end{figure}

We then tackle the problem to compute all matrix classes induced by matrices with two left columns. In view of the discussion at the end of Section~\ref{section the posets and bourn localizations}, it is enough to compute the poset $\mathsf{Mclex}_\ast[6,2,1]$, which becomes an easy task for the computer. The results are displayed in Figure~\ref{figure 6,2,1}. Let us notice that all matrices represented there have at most three rows, which indicates that the bound of $2m^m-(m-1)^m-1$ from Section~\ref{section the posets and bourn localizations} might be far from optimal.
\begin{figure}[htb!]
\centering
    \includegraphics[width=227pt]{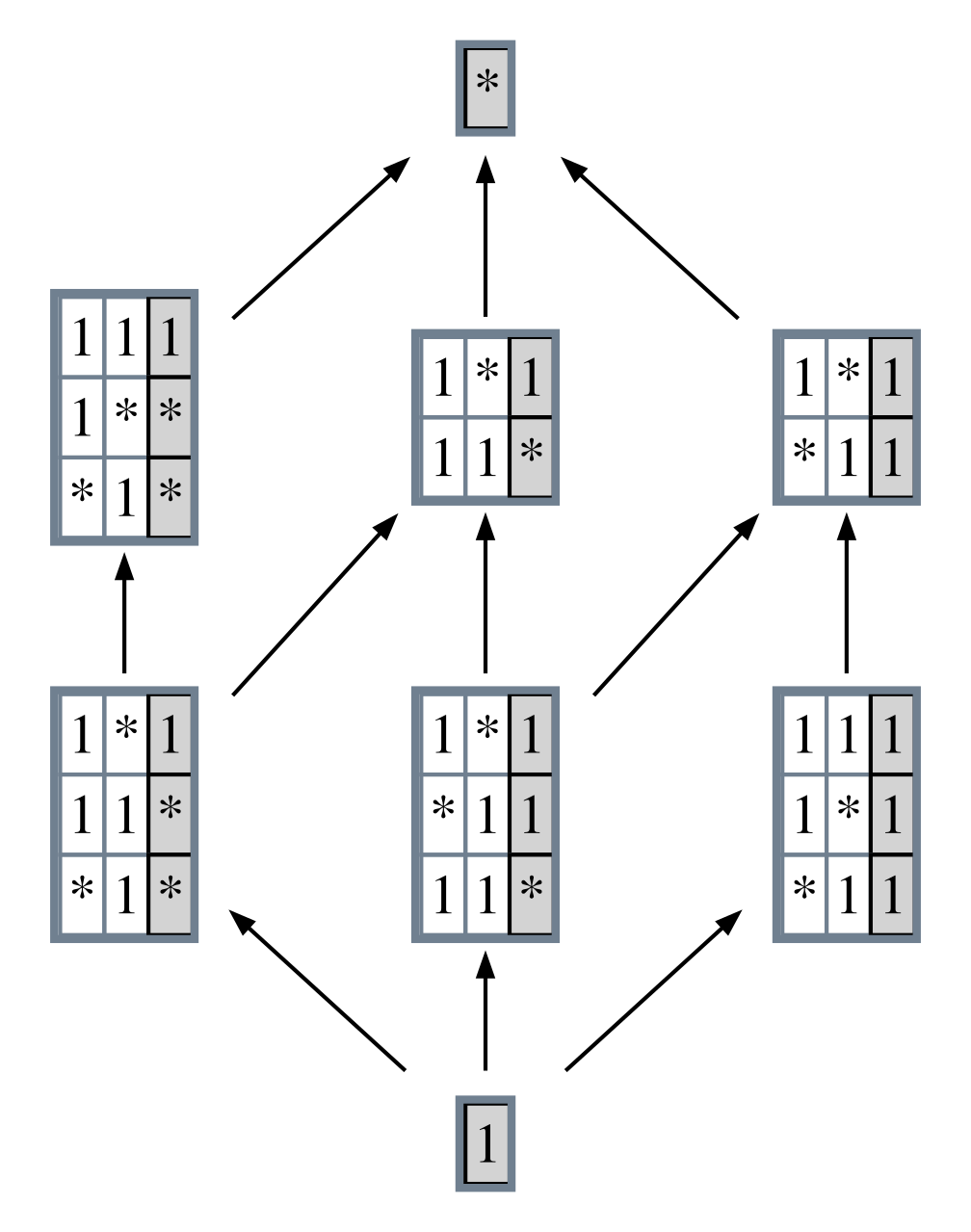}
    \caption{Hasse diagram of the poset $\mathsf{Mclex}_\ast[6,2,1]$.}
    \label{figure 6,2,1}
\end{figure}

Our next display is Figure~\ref{figure 3,6,1} which represents the poset $\mathsf{Mclex}_\ast[3,6,1]$. Since $6=(1+1)^3-2$, this poset contains all matrix classes generated by a matrix with three rows and one variable. The blue boxes group together the matrix classes which have same Bourn localization in $\mathsf{Mclex}$, i.e., same image under $\mathsf{Loc}\colon\mathsf{Mclex}_\ast\to\mathsf{Mclex}$. The words appearing in some of these boxes indicate the common Bourn localization of the matrix classes appearing in each of these boxes (see Example~\ref{example matrices}). These conventions will be used also in subsequent figures as well. By a \emph{minority category}, we mean here a finitely complete category which has $M$-closed relations for the non-pointed matrix
$$M=\left[\begin{array}{ccc|c} x_1 & x_2 & x_2 & x_1\\ x_2 & x_1 & x_2 & x_1 \\ x_2 & x_2 & x_1 & x_1 \end{array}\right],$$
extending the notion of a variety of universal algebras with a minority term. It can be easily proved using our algorithm (see below) that the first matrix in the eighth row (from the top) of Figure~\ref{figure 3,6,1} represents the matrix class of strongly unital categories, represented by another matrix in the third row of Figure~\ref{figure 2,3,2} (which means in particular that the $(2,3,2)$-canonical matrix for this matrix class is not the same as the $(3,6,1)$-canonical matrix). This leads to the following remark.

\begin{remark}
The collection of strongly unital categories is
$$\mathsf{mclex}_\ast\left\{\left[\begin{array}{ccc|c} x_1 & \ast & \ast & x_1\\ x_2 & x_1 & x_2 & x_1 \end{array}\right]\right\} = \mathsf{mclex}_\ast\left\{\left[\begin{array}{ccc|c} x_1 & x_1 & \ast & x_1\\ \ast & \ast & x_1 & x_1 \\ x_1 & \ast & x_1 & \ast \end{array}\right]\right\}.$$
\end{remark}

In order to illustrate the algorithm, let us display some tableaux that represent the proof of each inclusion of the above equality. These tableaux will be called \emph{$\text{lex}_{\ast}$-tableaux} (analogously to the `lex-tableaux' from~\cite{HoefnagelJacqminJanelidze}, where as usual `lex' abbreviates `left exact', a synonym of `finitely complete'). A proof of the inclusion
$$\mathsf{mclex}_\ast\left\{\left[\begin{array}{ccc|c} x_1 & \ast & \ast & x_1\\ x_2 & x_1 & x_2 & x_1 \end{array}\right]\right\} \subseteq \mathsf{mclex}_\ast\left\{\left[\begin{array}{ccc|c} x_1 & x_1 & \ast & x_1\\ \ast & \ast & x_1 & x_1 \\ x_1 & \ast & x_1 & \ast \end{array}\right]\right\}$$
can be represented by the $\text{lex}_{\ast}$-tableau
$$\begin{array}{ccc|c||ccc|c} 
     &      &      &      & x_1  & x_1  & \ast & x_1  \\
x_1  & \ast & \ast & x_1  & \ast & \ast & x_1  & x_1  \\
x_2  & x_1  & x_2  & x_1  & x_1  & \ast & x_1  & \ast \\
\hline
     &      &      &      & \ast &      &      &      \\
     &      &      &      & \ast &      &      &      \\
     &      &      &      & \ast &      &      &      \\
\hline
x_1  & \ast & x_1  & \ast & \ast &      &      &      \\
\ast & \ast & \ast & \ast & \ast &      &      &      \\
x_1  & \ast & \ast & x_1  & x_1  &      &      &      \\
\hline
\ast & x_1  & \ast & x_1  & x_1  &      &      &      \\
x_1  & \ast & \ast & x_1  & x_1  &      &      &      \\
x_1  & \ast & x_1  & \ast & \ast &      &      &      \\
\end{array}$$
where the double vertical line separates the two given matrices. In such a tableau, the matrix class induced by the matrix in the top left is being proved to be included in the matrix class induced by the matrix in the top right. The columns below the first horizontal line and on the right of the double vertical line are being added to the left columns of the second matrix by the algorithm. In the first step (just below the first horizontal line), we add a column of $\ast$'s as required by the algorithm. Then, the added columns come from row-wise interpretations of the first matrix; these row-wise interpretations are represented on the left of the double vertical line. Reaching the right column of the second matrix in the last step proves the desired inclusion of matrix classes. The following $\text{lex}_{\ast}$-tableau
$$\begin{array}{ccc|c||ccc|c} 
x_1  & x_1  & \ast & x_1  &      &      &      &      \\
\ast & \ast & x_1  & x_1  & x_1  & \ast & \ast & x_1  \\
x_1  & \ast & x_1  & \ast & x_2  & x_1  & x_2  & x_1  \\
\hline
     &      &      &      & \ast &      &      &      \\
     &      &      &      & \ast &      &      &      \\
\hline
\ast & \ast & x_1  & x_1  & x_1  &      &      &      \\
x_2  & \ast & x_2  & \ast & \ast &      &      &      \\
\hline
x_1  & x_1  & \ast & x_1  & x_1  &      &      &      \\
\ast & \ast & x_1  & x_1  & x_1  &      &      &      \\
\end{array}$$
represents a proof of
$$\mathsf{mclex}_\ast\left\{\left[\begin{array}{ccc|c} x_1 & x_1 & \ast & x_1\\ \ast & \ast & x_1 & x_1 \\ x_1 & \ast & x_1 & \ast \end{array}\right]\right\} \subseteq \mathsf{mclex}_\ast\left\{\left[\begin{array}{ccc|c} x_1 & \ast & \ast & x_1\\ x_2 & x_1 & x_2 & x_1 \end{array}\right]\right\}.$$

Figure~\ref{figure 3,6,1} illustrates the characterization of Mal'tsev categories in terms of the fibration of points from~\cite{Bourn1996} and~\cite{Janelidze2005} recalled in Section~\ref{section the posets and bourn localizations}. From Remark~\ref{remark restriction loc}, we know that $\mathsf{Loc}$ restricts to a surjective order-preserving function $\mathsf{Loc}\colon\mathsf{Mclex}_\ast[3,6,1]\twoheadrightarrow\mathsf{Mclex}[3,7,2]$ where $\mathsf{Mclex}[3,7,2]$ has been displayed in Figure~2 of~\cite{HoefnagelJacqminJanelidze} and contains all $\mathsf{mclex}\{M\}$ for a non-pointed matrix $M$ with three rows, two variables and at least one left column. Therefore, we can recover from Figure~\ref{figure 3,6,1} of the present paper that $\mathsf{Mclex}[3,7,2]$ has $13$ elements (represented by the $13$ blue boxes). Moreover, since $\mathsf{Loc}$ is an order preserving function, we know that an arrow in Figure~\ref{figure 3,6,1} displayed from a matrix in some box to a matrix in another box induces an arrow in the same direction between the corresponding elements of $\mathsf{Mclex}[3,7,2]$. However, Figure~\ref{figure 3,6,1} does not contain all the information about $\mathsf{Mclex}[3,7,2]$ as $\mathsf{Loc}$ does not reflect the order. For instance, any finitely complete arithmetical category is a minority category, but no arrow in Figure~\ref{figure 3,6,1} indicates this implication. The `missing arrow' appears in Figure~\ref{figure 3,3,2}.
\begin{figure}[phtb!]
\centering
    \includegraphics[width=484pt]{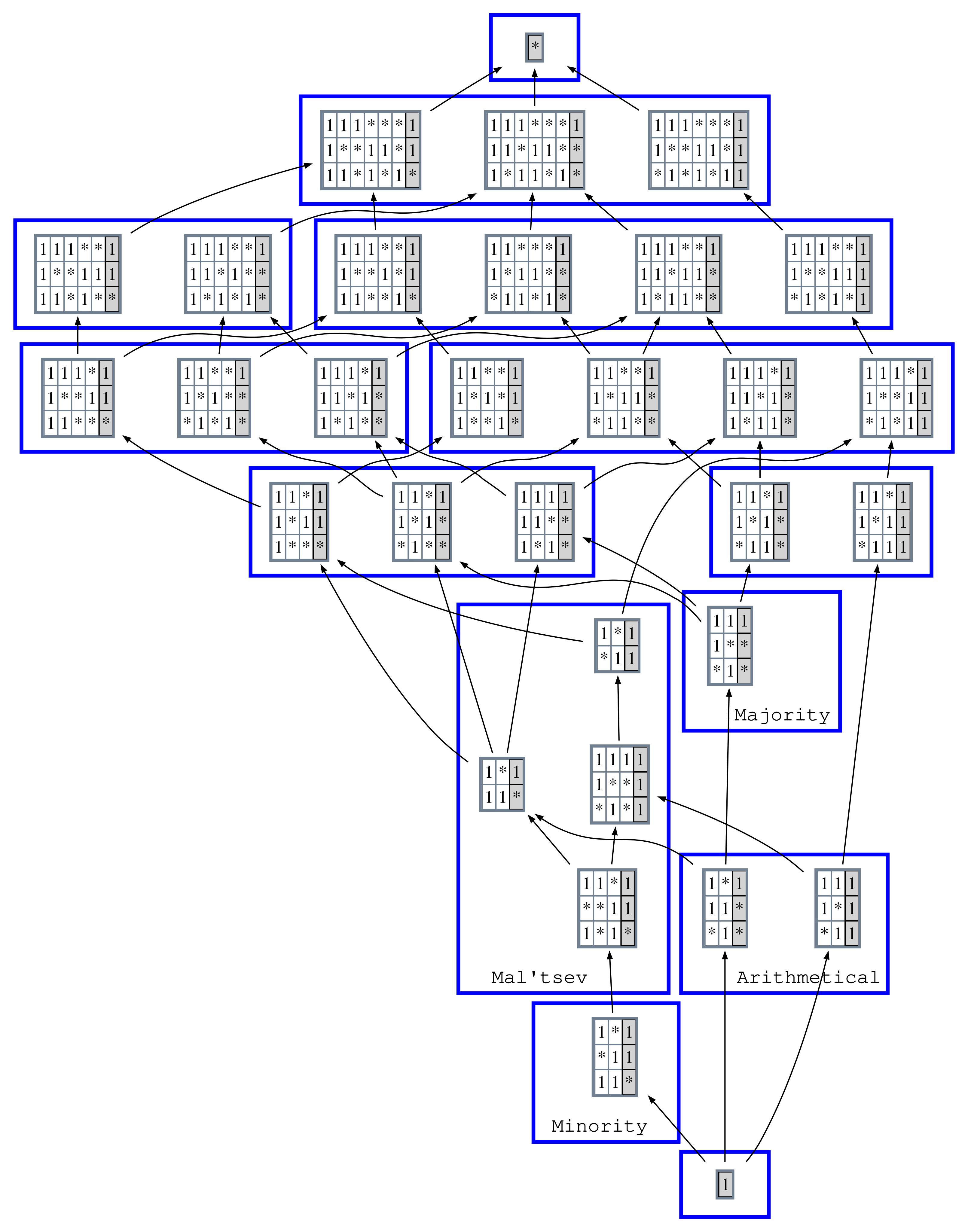}
    \caption{Hasse diagram of the poset $\mathsf{Mclex}_\ast[3,6,1]$.}
    \label{figure 3,6,1}
\end{figure}

Let us now consider Figure~\ref{figure 3,3,2} which represents $\mathsf{Mclex}_\ast[3,3,2]$. This figures illustrates Remark~\ref{remark localization classes} since each non-pointed matrix represented in that figure is the bottom element of the subposet of $\mathsf{Mclex}_\ast$ formed by the blue box containing it. Moreover, the top element of $\mathsf{Mclex}_\ast$ is the only one in its blue box and similarly for the bottom element of $\mathsf{Mclex}_\ast$.
\begin{figure}[phtb!]
\centering
    \includegraphics[width=484pt]{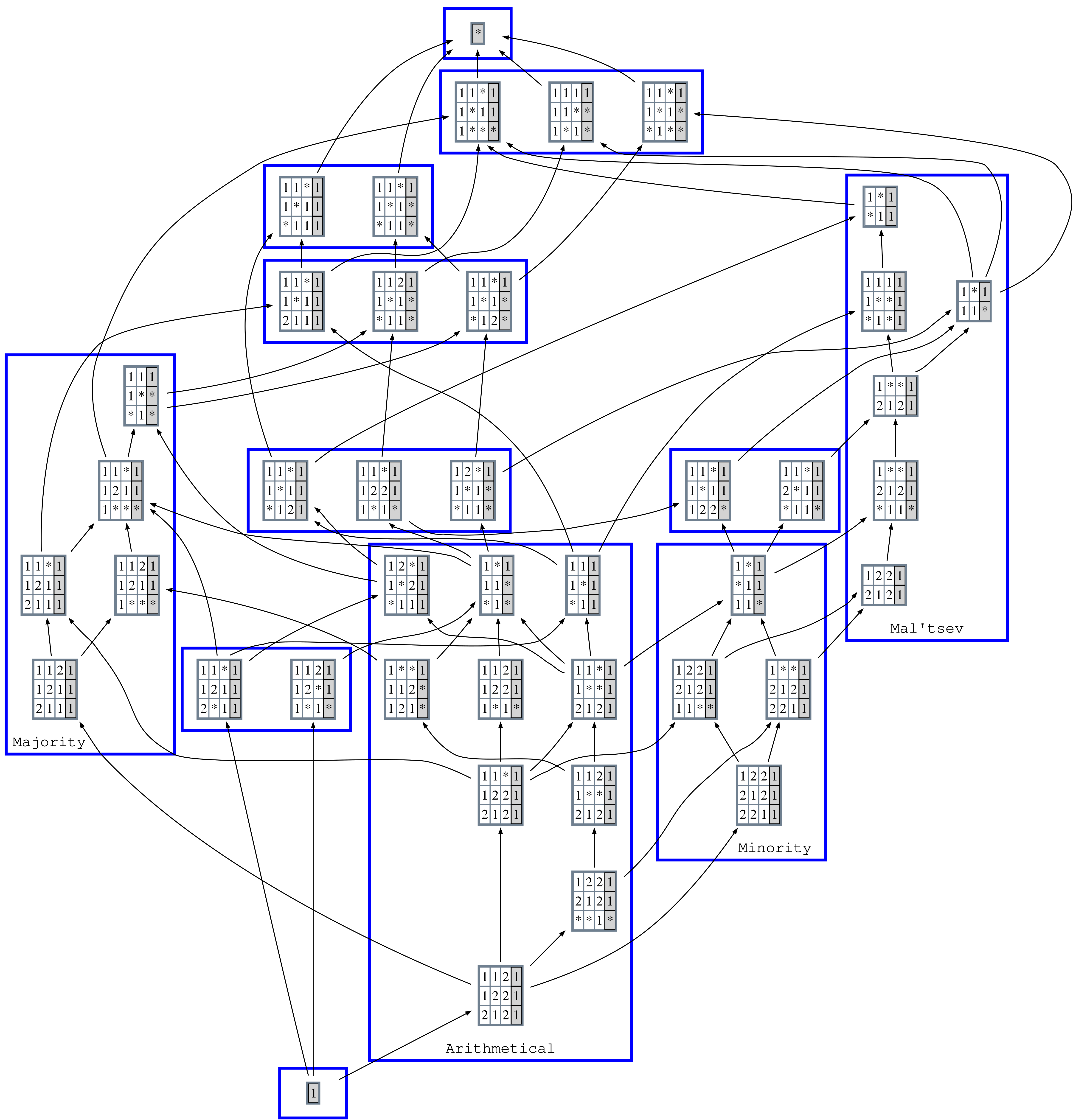}
    \caption{Representation of the poset $\mathsf{Mclex}_\ast[3,3,2]$.}
    \label{figure 3,3,2}
\end{figure}

We have also been able to compute $\mathsf{Mclex}_\ast[3,m,2]$ for $m\leqslant 6$, but $\mathsf{Mclex}_\ast[3,4,2]$ is already too big to reasonably fit on one page. We have instead produced Figure~\ref{figure sizes of 3xmx2} displaying the size of these posets.
\begin{figure}[htb!]
	\centering
$\begin{array}{|r|c|c|c|c|c|c|c|}
	\hline
	m= & 0 & 1 & 2 & 3 & 4 & 5 & 6\\ \hline
	 |\mathsf{Mclex}_\ast[3,m,2]|= & 2 & 2 & 8 &  42 &  217 &  1137 &  5100\\ \hline
	\end{array}$
	\caption{Size of $\mathsf{Mclex}_\ast[3,m,2]$ for $m\leqslant 6$.}
	\label{figure sizes of 3xmx2}
\end{figure}

Figure~\ref{figure 4,3,1} displays the poset $\mathsf{Mclex}_\ast[4,3,1]$. As we can see, even when considering one variable matrices, the poset can quickly get complex.
\begin{figure}[phtb!]
\centering
    \includegraphics[width=484pt]{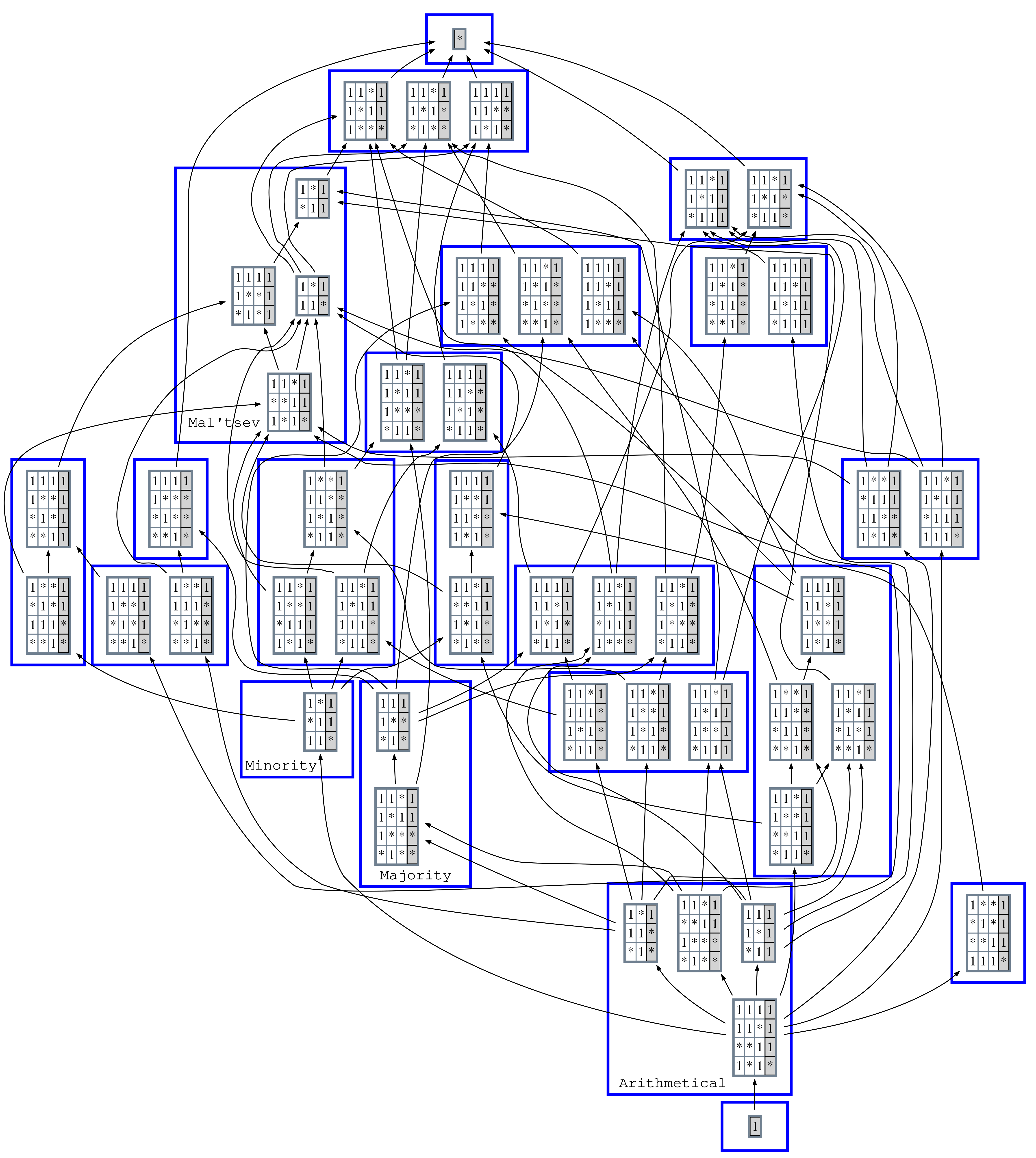}
    \caption{Hasse diagram of the poset $\mathsf{Mclex}_\ast[4,3,1]$.}
    \label{figure 4,3,1}
\end{figure}

We have also computed the posets $\mathsf{Mclex}_\ast[4,m,1]$ for any value of~$m$. However, for $m=4$, this poset already contains $156$ elements and cannot be represented with enough readability on a single page. We have instead displayed the size of these posets on Figure~\ref{figure sizes of 4xmx1}. Let us remark here that the matrix classes of unital, strongly unital and subtractive categories respectively are elements of $\mathsf{Mclex}_\ast[4,14,1]$, while the matrix classes of Mal'tsev, majority, arithmetical and minority pointed categories respectively are not elements of $\mathsf{Mclex}_\ast[4,14,1]$.
\begin{figure}[hbt!]
	\centering
	\begin{tikzpicture}[yscale=0.6, xscale=0.6]
	\begin{axis}[xlabel={$m$},ylabel={$|\mathsf{Mclex}_\ast[4,m,1]|$}]
	\addplot[color=black,mark=*] coordinates {
		(0,2)
		(1,2)
		(2,8)
		(3,48)
		(4,156)
		(5,453)
		(6,1066)
		(7,1953)
		(8,2841)
		(9,3502)
		(10,3822)
		(11,3957)
		(12,4007)
		(13,4023)
		(14,4027)
		(15,4027)
	};
	\end{axis}
	\end{tikzpicture}
	$\quad\quad$
	\begin{tikzpicture}[yscale=0.7, xscale=0.7]
    \begin{axis}[xlabel={$m$},ylabel={$|\mathsf{Mclex}_\ast[4,m,1]|-|\mathsf{Mclex}_\ast[4,m-1,1]|$}]
	\addplot[color=black,mark=*] coordinates {
		(1,0)
		(2,6)
		(3,40)
		(4,108)
		(5,297)
		(6,613)
		(7,887)
		(8,888)
		(9,661)
		(10,320)
		(11,135)
		(12,50)
		(13,16)
		(14,4)
		(15,0)
	};
	\end{axis}
    \end{tikzpicture}
	
	\begin{adjustwidth}{-36pt}{-36pt}
	$\begin{array}{|r|c|c|c|c|c|c|c|c|c|c|c|c|c|c|c|c|}
	\hline
	m= & 0 & 1 & 2 & 3 & 4 & 5 & 6 & 7 & 8 & 9 & 10 & 11 & 12 & 13 & 14+\\ \hline
	 |\mathsf{Mclex}_\ast[4,m,1]|= & 2 & 2 & 8 &  48 & 156 & 453 & 1066 & 1953 & 2841 & 3502 & 3822 & 3957 & 4007 & 4023 & 4027\\ \hline
	\end{array}$
	\end{adjustwidth}
	\caption{Size of $\mathsf{Mclex}_\ast[4,m,1]$.}
	\label{figure sizes of 4xmx1}
\end{figure}

We conclude the paper with the representation of some subposets of some $\mathsf{Mclex}_\ast[n,m,k]$ formed by elements with the same Bourn localization. We start with Figure~\ref{figure delocalization Maltsev 3,4,2} which displays the subposet of $\mathsf{Mclex}_\ast[3,4,2]$ formed by matrix classes whose Bourn localization is the collection of Mal'tsev categories. The analogous poset defined by replacing $\mathsf{Mclex}_\ast[3,4,2]$ with $\mathsf{Mclex}_\ast[3,5,2]$ contains $268$ elements and is too big to be reasonably displayed on a single page. As we already know from Remark~\ref{remark localization classes}, the matrix class of Mal'tsev pointed categories is the bottom element of this subposet of $\mathsf{Mclex}_\ast[3,4,2]$. The two maximal elements of this subposet are the matrix classes of unital and subtractive categories. They intersection (i.e., the matrix class of strongly unital categories) is represented by the left matrix in the seventh row of that figure. Figure~\ref{figure delocalization Maltsev 3,4,2} gives thus many more characterizations of Mal'tsev categories in terms of the fibres of the fibration of points in the style of~\cite{Bourn1996,Janelidze2005}.
\begin{figure}[phtb!]
\centering
    \includegraphics[width=484pt]{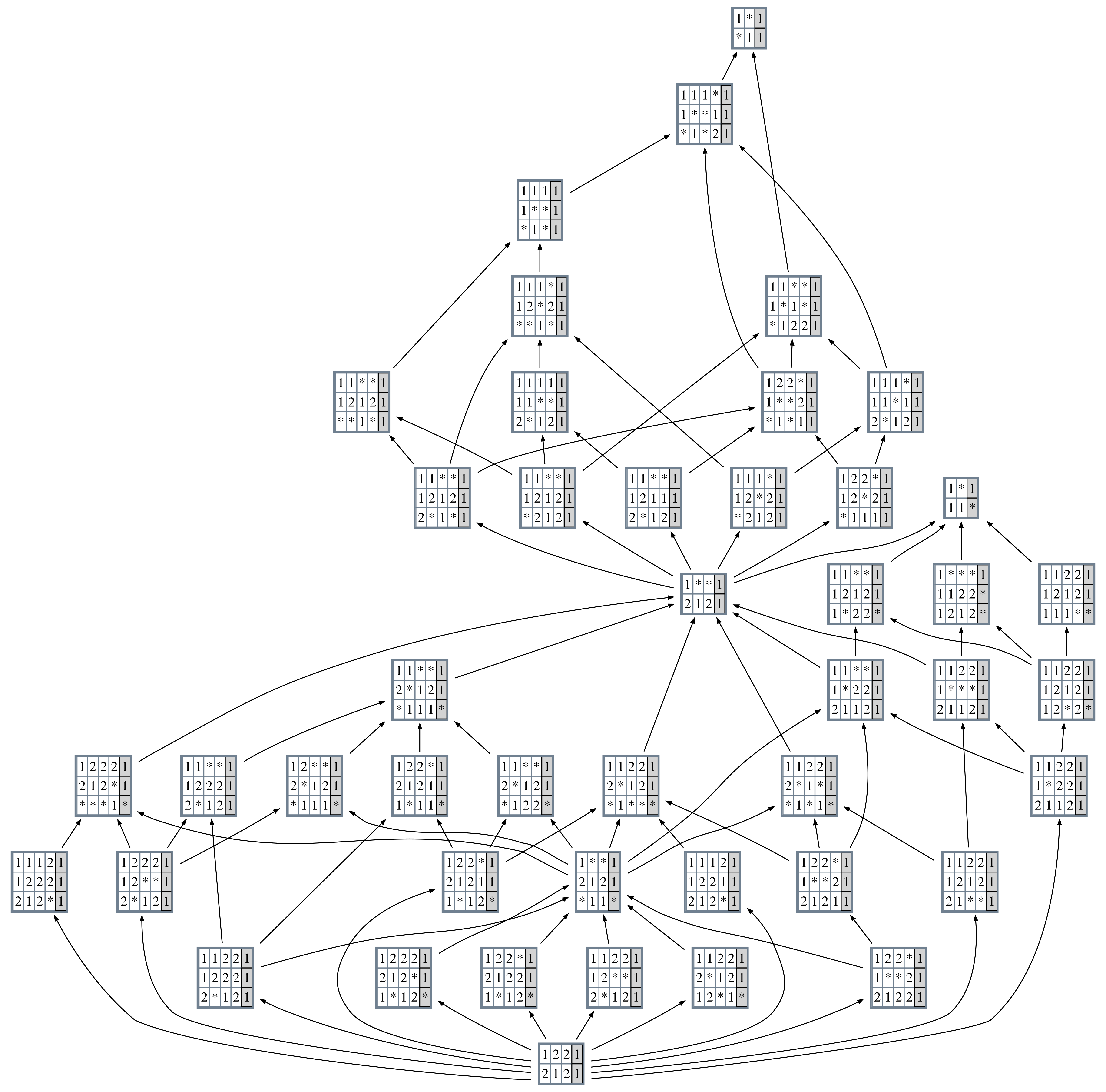}
    \caption{Hasse diagram of the subposet of $\mathsf{Mclex}_\ast[3,4,2]$ formed by matrix classes whose Bourn localization is the collection of Mal'tsev categories.}
    \label{figure delocalization Maltsev 3,4,2}
\end{figure}

Figure~\ref{figure delocalization Maltsev 4,14,1} represents the subposet of $\mathsf{Mclex}_\ast[4,14,1]$ formed by matrix classes whose Bourn localization is the collection of Mal'tsev categories. Since the matrix class of pointed Mal'tsev categories does not belong to $\mathsf{Mclex}_\ast[4,14,1]$, it is not represented on this figure. The bottom element of this subposet of $\mathsf{Mclex}_\ast[4,14,1]$ is the matrix class of strongly unital categories, represented here by its $(3,3,1)$-canonical matrix. Again, the two maximal elements of this subposet are the matrix classes of unital and subtractive categories. Let us remark that, in both Figures~\ref{figure delocalization Maltsev 3,4,2} and~\ref{figure delocalization Maltsev 4,14,1}, the number of elements between the matrix class of strongly unital categories and the matrix class of unital categories is much bigger than the number  of elements between the matrix class of strongly unital categories and the matrix class of subtractive categories. We can also observe that, although we considered all matrix classes in $\mathsf{Mclex}_\ast[4,14,1]$ whose Bourn localization is the collection of Mal'tsev categories, they are all represented by matrices with at most seven left columns.
\begin{figure}[phtb!]
\centering
    \includegraphics[width=484pt]{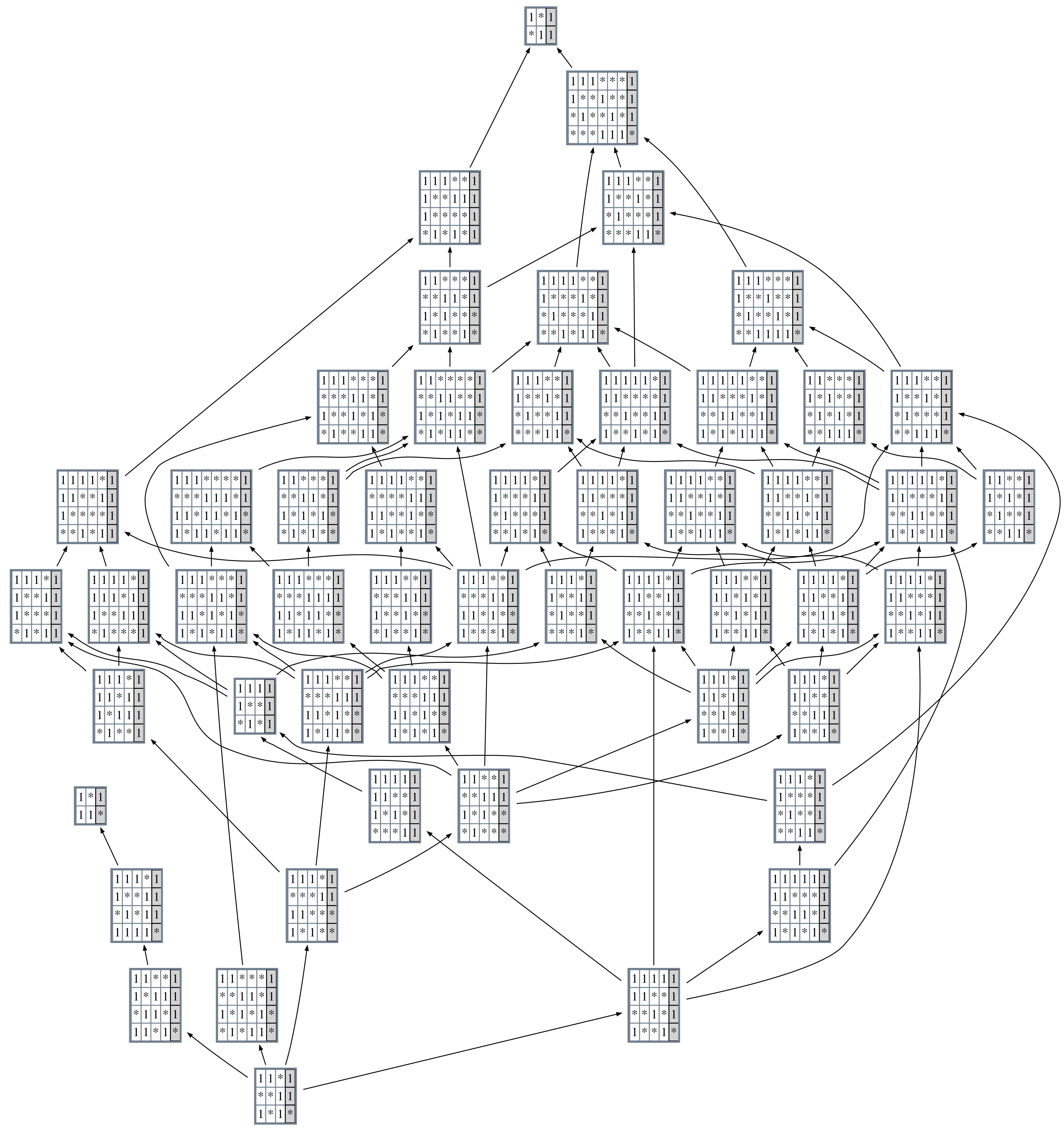}
    \caption{Hasse diagram of the subposet of $\mathsf{Mclex}_\ast[4,14,1]$ formed by matrix classes whose Bourn localization is the collection of Mal'tsev categories.}
    \label{figure delocalization Maltsev 4,14,1}
\end{figure}

There are $123$ matrix classes in $\mathsf{Mclex}_\ast[3,6,2]$ whose Bourn localization is the collection of finitely complete arithmetical categories, which is too many to draw them on a single page with enough readability. Instead, we show in Figure~\ref{figure delocalization arithmetical 3,5,2} the subposet of $\mathsf{Mclex}_\ast[3,5,2]$ formed by those matrix classes whose Bourn localization is the collection of finitely complete arithmetical categories. As expected, the bottom element of this subposet is the matrix class of arithmetical finitely complete pointed categories. There are only four matrix classes in $\mathsf{Mclex}_\ast[4,14,1]$ whose Bourn localization is the collection of finitely complete arithmetical categories. Since they can all be represented by a matrix in $\M(4,3,1)$, they appear in the same blue box in Figure~\ref{figure 4,3,1}.
\begin{figure}[phtb!]
\centering
    \includegraphics[width=484pt]{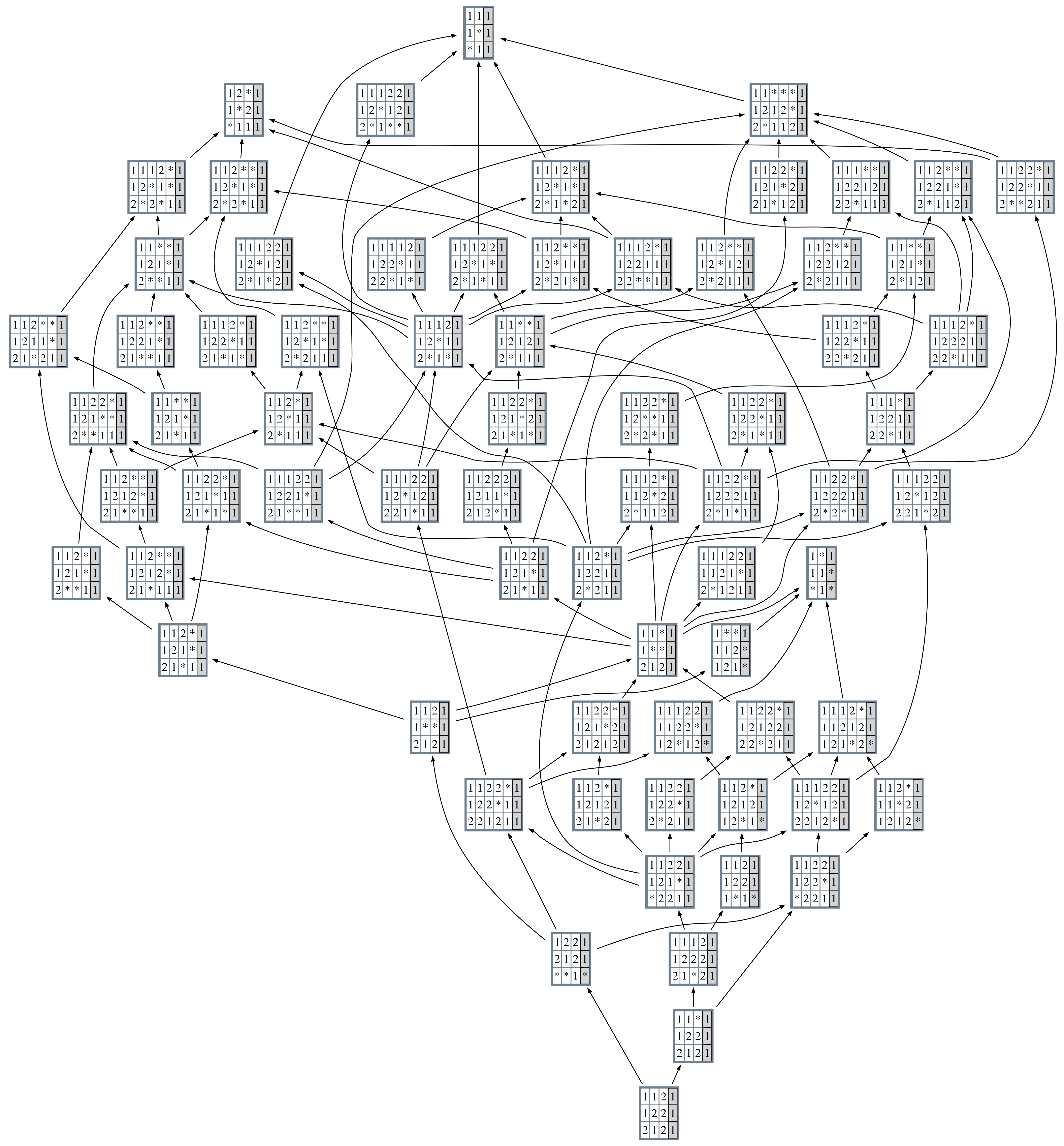}
    \caption{Hasse diagram of the subposet of $\mathsf{Mclex}_\ast[3,5,2]$ formed by matrix classes whose Bourn localization is the collection of (finitely complete) arithmetical categories.}
    \label{figure delocalization arithmetical 3,5,2}
\end{figure}

Figure~\ref{figure delocalization majority 3,5,2} represents the subposet of $\mathsf{Mclex}_\ast[3,5,2]$ formed by matrix classes whose Bourn localization is the collection of majority categories. The subposet of $\mathsf{Mclex}_\ast[3,6,2]$ defined analogously having $89$ elements, we choose not to represent it here for the sake of readability. The subposet of $\mathsf{Mclex}_\ast[4,14,1]$ formed by matrix classes whose Bourn localization is the collection of majority categories has only $3$ elements and is displayed in Figure~\ref{figure delocalization majority 4,14,1}. Notice that all the matrix classes appearing there are represented by a matrix with at most four left columns.
\begin{figure}[phtb!]
\centering
    \includegraphics[width=475pt]{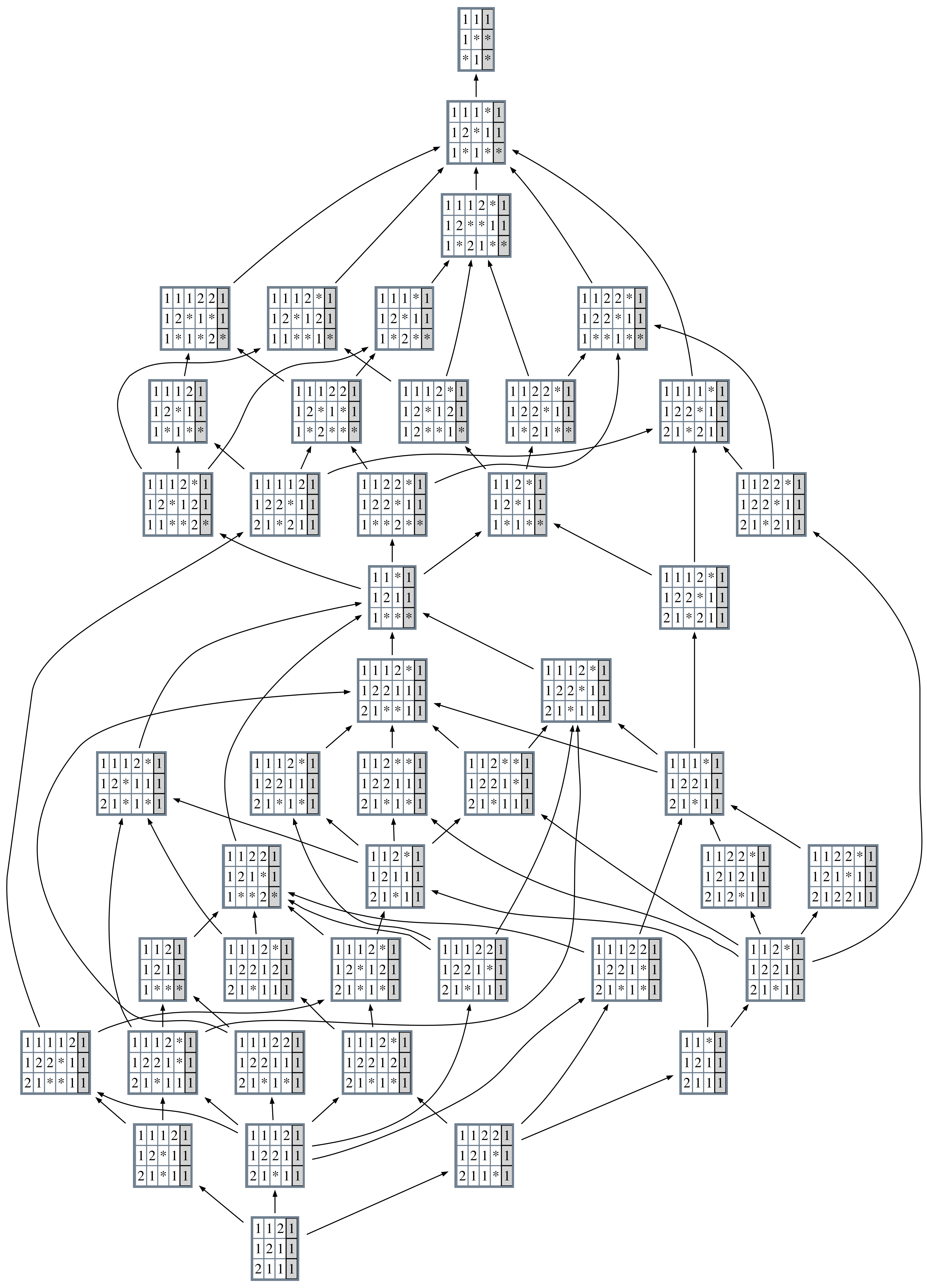}
    \caption{Hasse diagram of the subposet of $\mathsf{Mclex}_\ast[3,5,2]$ formed by matrix classes whose Bourn localization is the collection of majority categories.}
    \label{figure delocalization majority 3,5,2}
\end{figure}

\begin{figure}[htb!]
\centering
    \includegraphics[width=70pt]{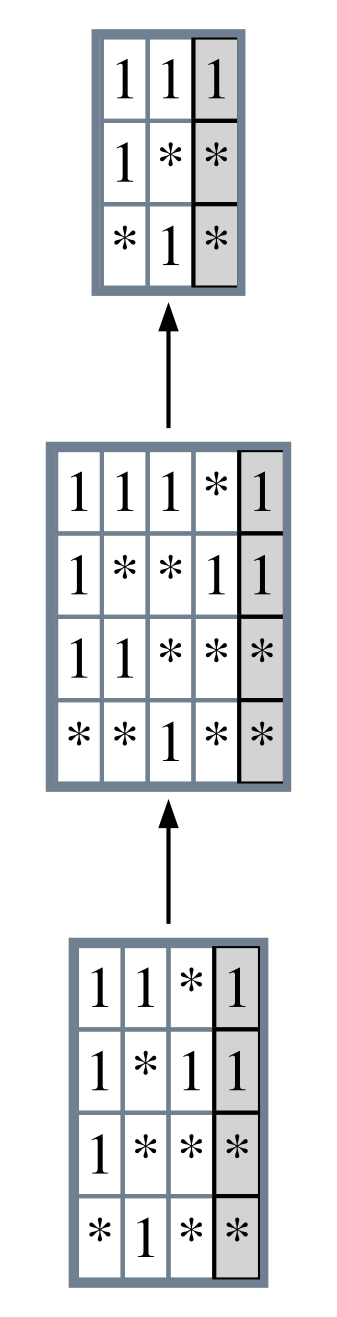}
    \caption{Hasse diagram of the subposet of $\mathsf{Mclex}_\ast[4,14,1]$ formed by matrix classes whose Bourn localization is the collection of majority categories.}
    \label{figure delocalization majority 4,14,1}
\end{figure}

We have found very few matrix classes whose Bourn localization is the collection of minority categories. There is only one which belongs to $\mathsf{Mclex}_\ast[4,14,1]$: it is represented by a matrix in $\M(3,2,1)$ in Figure~\ref{figure 4,3,1}. There are $12$ which belong to $\mathsf{Mclex}_\ast[3,6,2]$ and the corresponding poset has been represented in Figure~\ref{figure delocalization minority 3,6,2}. As we can notice, all these matrix classes actually belong to~$\mathsf{Mclex}_\ast[3,5,2]$.
\begin{figure}[htb!]
\centering
    \includegraphics[width=400pt]{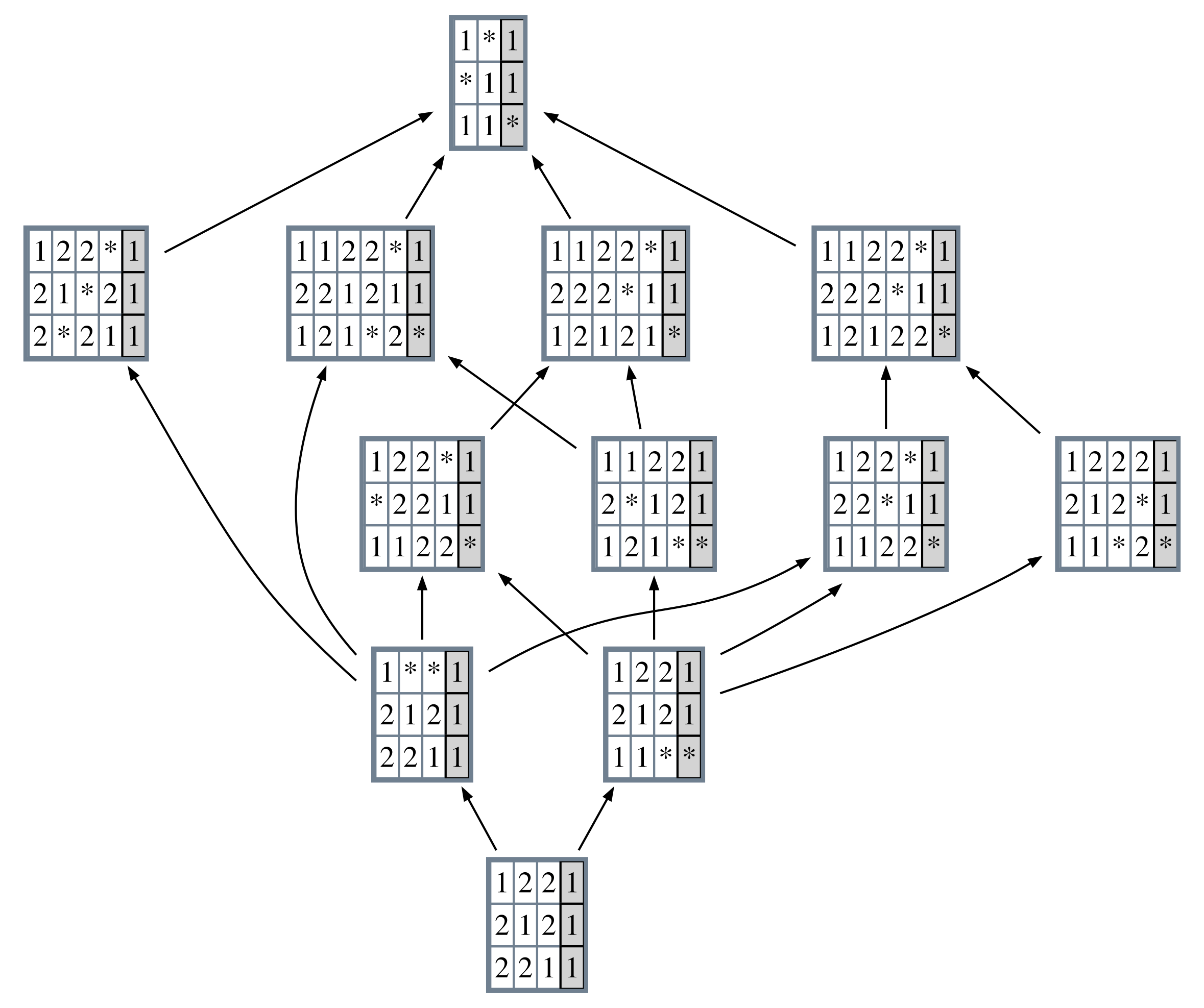}
    \caption{Hasse diagram of the subposet of $\mathsf{Mclex}_\ast[3,6,2]$ formed by matrix classes whose Bourn localization is the collection of minority categories.}
    \label{figure delocalization minority 3,6,2}
\end{figure}

\FloatBarrier

\subsection*{Acknowledgments}

The authors would like to warmly thank Zurab Janelidze for interesting discussions on the topic presented here. They are also grateful to the anonymous referee for his/her remarks to improve the reader's experience. The first author is grateful to the National Graduate Academy for Mathematical and Statistical Sciences (NGA(MaSS)) for its generous financial support. The second author also acknowledges the FNRS for its generous support.

\subsection*{Data availability}

Data sharing not applicable to this article as no datasets were generated or analysed during the current study.


\vspace{30pt}
\begin{tabular}{rl}
Email: & mhoefnagel@sun.ac.za\\
& pierre-alain.jacqmin@uclouvain.be
\end{tabular}

\end{document}